\documentclass[11pt]{article}

\usepackage{amsmath}%
\usepackage{amsfonts}%
\usepackage{amssymb}%
\usepackage{graphicx}
\usepackage[T2A]{fontenc}
\usepackage[utf8]{inputenc}

\usepackage{amsthm} 
\usepackage{hyperref} 
\usepackage{enumerate}

\usepackage[all,cmtip]{xy}

\usepackage[small,nohug,heads=littlevee]{diagrams}
\diagramstyle[labelstyle=\scriptstyle]

\DeclareMathOperator{\Ker}{Ker} 
 
\DeclareMathOperator{\cl}{cl}

\DeclareMathOperator{\Int}{int}
\DeclareMathOperator{\Dom}{dom}
\DeclareMathOperator{\Sp}{span}
\DeclareMathOperator{\icr}{icr}
\DeclareMathOperator{\co}{co}

\newtheorem{theorem}{Theorem}
\numberwithin{theorem}{section}
\theoremstyle{plain}
\theoremstyle{definition}\newtheorem{example}{Example}

\newtheorem{corollary}{Corollary}[theorem]

\theoremstyle{definition}\newtheorem{definition}{Definition}
\theoremstyle{definition} \newtheorem{cntex}{Counterexample}
\theoremstyle{definition}
\newtheorem{lemma}{Lemma}
\numberwithin{lemma}{section}

\newtheorem{proposition}{Proposition}
\numberwithin{proposition}{section}

\theoremstyle{definition}
\numberwithin{equation}{section}

\begin{document}

\title{Algebraic Lipschitz and Subdifferential Calculus in Vector Spaces}
\author{Dmytro Voloshyn \thanks{Institute for Applied System Analysis, National Technical University of Ukraine “Igor Sikorsky Kyiv Polytechnic Institute”, Peremogy ave., 37, build, 35, 03056, Kyiv, Ukraine, dmytro.voloshyn@yahoo.com
}}
\maketitle
\begin{abstract}
The main contribution of this paper is that every convex function with non-empty relative algebraic interior of its domain is Lipschitz and subdifferentiable in some algebraic sense without any additional topological constraints. The proposed approach uses slightly modified Clarke's subdifferential for functions defined on a convex symmetric set and Lipschitz with respect to a Minkowski functional. Following this, Clarke's subdifferential calculus is generalized to vector spaces and, where continuity properties are claimed, to topological vector spaces. 
\end{abstract}

\section{Introduction}
Let $X$ be a real vector space, $S\subset X$ be a convex subset of $X$, $\varphi : S \rightarrow \mathbb{R}$ be a convex function. We denote by $A-B$ the algebraic subtraction of sets $\{A,B\} \subset X$, i.e. $A-B:=\{a - b : a \in A, \ b \in B\}.$
The linear hull of a set $A\subset X$ is denoted as $\Sp A$. Recall that \textit{the relative algebraic interior} of the convex set $S$ is the set defined by
$$ \icr S := \{x \in S : \text{for every } y \in \Sp (S-S)\text{ there exists }t>0 \text{ such that } x+ty  \in S\},$$ 
\textit{the Minkowski functional} (or \textit{the Minkowski gauge}) of the set $S$ is the function $\mu_S:X\rightarrow\mathbb{R}\cup\{+\infty\}$ defined by
$$ \mu_{S} (x) := \inf \{ t > 0 : x \in tS\}, \ x \in X \text{ (we put }\inf \emptyset := +\infty\text{).}$$ 
Note that for any $b \in S$ the equality $\Sp(S-S) = \Sp (S-b)$ holds. We say that the convex set $S$ is \textit{symmetric with respect to a point} $x \in S$ if $(S-x)=-(S-x)$.
The notation $\varphi|_{A}$ is used for the restriction of the function $\varphi$ to a subset $A\subset S$.
The algebraic dual of the space $X$ is denoted as $X^\prime$. If the space $X$ is supplied with a vector topology, then its topological dual space is denoted as $X^*$. The standard pairing between $X^*$ and $X$ is denoted as $\left\langle,\right\rangle$, i.e. $\left\langle \zeta, x \right\rangle = \zeta (x), \ \zeta \in X^*, \ x \in X$. The notation $\sigma(X^*, X)$ is used for the smallest vector topology in $X^*$ such that for every $x \in X$ a linear functional of the form $\left\langle \cdot, x\right\rangle$ is continuous. The closure and interior operators are denoted as $\cl$ and $\Int$. With respect to the Minkowski functional $\mu_S$, the balls are denoted as \text{$B_{\mu_S}(x_0,\varepsilon):=\{x \in \Sp (S-S) : \mu_S(x-x_0) < \varepsilon\}$.} In case the space $X$ is supplied with a normed structure, the notation $B(x_0,\varepsilon)$ is used for the open balls. For a vector topology in $X$ and for a function $f: U \rightarrow \mathbb{R}$, where $U\subset X$ is an open set, \textit{the generalized directional derivative} is defined by
$$f ^\circ(x,v) := \limsup_{\substack{y\rightarrow x, \\ t\rightarrow 0+}}\frac{f(y+tv)-f(y)}{t},\ \ x \in U, \ v \in X,$$
and \textit{the Clarke subdifferential }is defined by
$$ \partial_C f(x) := \{\zeta \in X^* : f^\circ(x,v) \geq \left\langle \zeta, v \right\rangle \text{ for all } v \in X\}, \ \ x \in U,$$
\textit{The Fenchel subdifferential} of the convex function $\varphi$ is defined by
$$ \partial \varphi(x) := \{\zeta \in X^* : \varphi(y) - \varphi(x) \geq \left\langle \zeta, y-x \right\rangle \text{ for all } y \in S\},\ \ x \in S.$$
For the properties of both the Fenchel subdifferential and the Clarke subdifferential one may refer to Clarke \cite{clarke_book} or Z\u alinescu \cite{zalinescu}.

The composition of functions $f_1 : A \rightarrow B$ and $f_2 : B \rightarrow C$, where $A$, $B$ and $C$ are arbitrary sets, is denoted as $f_1\circ f_2$, i.e. $f_1\circ f_2 (x) = f_1(f_2(x))$, $x \in A$. For a function \text{$f : X \rightarrow \mathbb{R}\cup\{+\infty, -\infty\}$} we denote \text{$\Dom f := \{ x \in X : f (x) < +\infty\}$}. Let $T_1$ and $T_2$ be topological spaces, \text{$M:T_1 \rightrightarrows T_2$} be a multivalued map. Recall that the map $M$ is called \textit{upper semi-continuous} (u.s.c.) if for every non-empty closed subset $Q$ of the space $T_2$ the set
$$\{y\in T_1 : M(y) \cap Q \neq \emptyset\}$$
is closed in $T_1$. A function $g :T_1 \rightarrow \mathbb{R}$ is called \textit{upper semi-continuous} (u.s.c.) if for every point $x \in T_1$ and for every net $\{x_\alpha\}_{\alpha \in A}$ in $T_1$ converging to $x$ the next inequality holds:
$$ \limsup_{x_\alpha \rightarrow x}g(x_\alpha) \leq g(x).$$

In order to ensure that the subdifferential of a convex function is non-empty, one should use topological assumptions on the function. These assumptions may be found, e.g., in Amara and Ciligot-Travain \cite{amara}, Br\o ndsted and Rockafellar  \cite{bronsted}, Laghdir \cite{laghdir}, Moussaoui and Volle \cite{mouss}, Simons \cite{simons} and in Z\u alinescu \cite{zalinescu}. In this paper it is proposed a method of finding a non-empty subdifferential of a convex function with only assumption that its domain has non-empty relative algebraic interior. The idea is quite simple. Let $\varphi : S \rightarrow \mathbb{R}$ be a convex function defined on a convex set $S$ such that $\icr S \neq \emptyset$. In Section \ref{section_convex_case} we construct a certain convex set $C_A \subset S$, the ``capacity'' of which depends on a real number $A$; the set $C_A$ is symmetric with respect to a point $x_0 \in C_A$ and $\varphi (x) \leq A$ for all $x \in C_A$. Then we show that in some sense the function $\varphi$ is locally Lipschitz on $\icr S$. In \text{Section \ref{section_clarke}} we consider a function that is defined on a convex symmetric set and Lipschitz with respect to the Minkowski functional of its domain. We translate the Clarke subdifferential calculus for such a type of functions from the case of Banach spaces. In \text{Section \ref{section_convex_subd}} we merge the results from \text{Section \ref{section_convex_case}} and \text{Section \ref{section_clarke}}, and thus we obtain a subdifferential calculus for any convex function in a general vector space with non-empty relative algebraic interior of its domain. In \text{Section \ref{section_convex_subd}} we also discuss a way of defining the Fenchel subdifferential and show that in our case it is less convenient to deal with this subdifferential. In \text{Section \ref{examples}} we provide simple illustrations and counterexamples to some theorems of the article; particularly, the main idea is illustrated in Example \ref{example_conv}. The resulted subdifferential of the function $\varphi$ is written in the following form:
$$\partial_C^\mu \varphi (x) := \{\zeta \in \Sp (C_A - C_A)^\prime : \varphi^\prime (x,v) \geq \left\langle \zeta, v \right\rangle \text{ for all } v \in \Sp(C_A-C_A) \}, \  \ x \in \icr S,$$
where $\mu$ indicates the Minkowski functional of a certain translation of the set $C_A$ to the origin.

\section{Algebraic properties of convex functions}\label{section_convex_case}
Let $S$ be a non-empty convex subset of a real vector space such that $\icr S \neq \emptyset$ and $\varphi :S \rightarrow \mathbb{R}$ be a convex function. In Subsection \ref{ca} we construct a convex set $C_A\subset S$ that is symmetric with respect to a point $x_0 \in C_A$ and such that $\varphi|_{C_A}$ is bounded above. In Subsection \ref{lipschitz} we study a Lipschitz property of the function $\varphi$ and then we establish a locally Lipschitz property on $\icr S$. The precise definitions of the Lipschitz properties are given in Subsection \ref{lipschitz}.
\subsection{The construction of $C_A$}\label{ca}
For an arbitrary point $x_0 \in \icr S$, we choose a real number $A \geq \varphi(x_0)$ and we denote
\begin{equation}\label{eq:s_a}
S_A := \{x \in S: \varphi(x) \leq A\}.
\end{equation}
\begin{lemma}\label{lemma_icr}
Let $S_A$ be the set defined in \eqref{eq:s_a}. Then $x_0 \in \icr S_A$.
\end{lemma}
\begin{proof}Since $S_A$ is convex, every element $x \in  \Sp (S_A-S_A)$ can be represented in the form
$$
x = \alpha v - \beta u, \ \text{ where } \alpha >0, \ \beta > 0 \text{ and }  \{v,u\}\subset S_A-x_0.
$$
Let us put $t := \frac{1}{2(\alpha+\beta)}$. Then
$$\varphi(x_0 + tx) = \varphi(x_0 + \alpha tv + \beta t (-u)) = \varphi((1-\alpha t - \beta t) x_0 + \alpha t (v + x_0) + \beta t (x_0-u)) \leq$$$$\leq  (1-\alpha t - \beta t) A+ \alpha t A + \beta t A = A, $$
therefore $x_0 + tx \in S_A$ and thus $x_0 \in \icr S_A$.
\end{proof}
Denote
\begin{equation}\label{eq:c_a}
C_A := \{x \in S_A : \text{ there exists } \alpha > 0 \text{ such that } x_0 + \alpha (x-x_0) \in S_A \text{ and } x_0 + (-\alpha)(x-x_0) \in S_A\},
\end{equation}
where $S_A$ is the set defined in \eqref{eq:s_a}. The set $C_A$ by the definition is convex and symmetric with respect to the point $x_0$. If the set $S_A - x_0$ is symmetric, then $C_A = S_A$ (for any element from $S_A$ one may put $\alpha := 1$).
The ``capacity'' of the set $C_A$ can be measured by the linear span of $C_A-C_A$
(or by the affine hull of $C_A$). 
\begin{proposition} Let $S_A$ and $C_A$ be the sets defined in \eqref{eq:s_a} and \eqref{eq:c_a}. Then
\begin{equation} 
\Sp (C_A - C_A) = \Sp (S_A-S_A).\label{eq:spanseq}
\end{equation}
\end{proposition}
\begin{proof}
Indeed, since $C_A \subset S_A$, we see that $\Sp (C_A - C_A) \subset \Sp (S_A-S_A)$. Let $x \in \Sp (S_A-S_A)$. It follows from Lemma \ref{lemma_icr} that $x_0 \in \icr S_A$; hence, there exists $t > 0$ such that $x_0 + t(x-x_0) \in S_A$ and $x_0 + t (x_0-x) \in S_A$. Therefore $x \in C_A$ and thus $\Sp (C_A - C_A) = \Sp (S_A-S_A)$.\qedhere
\end{proof}
Particularly if $\varphi$ is bounded above, then the equality \eqref{eq:spanseq} may be written as
$$ \Sp (C- C) = \Sp (S-S), \text{  where } C:= C_{\sup\limits_{x \in S}\varphi(x)}.$$
\subsection{The $\mu$-Lipschitz property}\label{lipschitz}
In this subsection we study a Lipschitz property of convex functions in general vector spaces.

Let $C$ be a convex subset of a real vector space such that $C$ is symmetric with respect to a point $p \in C$ and let $\mu$ denotes the Minkowski functional of the set $C-p$.
The Minkowski functional of any absorbing symmetric convex set is a seminorm (see, e.g., Rudin \cite[Theorem 1.35]{rudin}). In the next lemma we show that $C-p$ is absorbing in $\Sp (C-p)$. Therefore $\mu$ only takes finite values on the space $\Sp(C-p)$, and moreover $\mu$ is a seminorm in $\Sp(C-p)$.
\begin{lemma}\label{absorb}
Let $C$ be a convex subset of a real vector space such that $C$ is symmetric with respect to a point $p \in C$. Then the set $C-p$ is absorbing in $\Sp (C-p)$.
\end{lemma}
\begin{proof}
Let $x \in \Sp (C-p)\setminus \{0\}$. Since $C$ is convex, the element $x$ can be represented in the form
$$x = \alpha v - \beta u, \ \text{ where } \alpha >0, \ \beta > 0 \text{ and }  \{v,u\}\subset C-p.$$
Since $C-p$ is symmetric, $-u \in C-p$, hence
$$x = \alpha v + \beta (-u) = \frac{1}{\alpha + \beta} \left(\frac{\alpha}{\alpha+\beta} v + \frac{\beta}{\alpha+\beta}(-u)\right) \in \frac{1}{\alpha+\beta} (C-p) \subset C-p,$$
i.e. $C-p$ is absorbing in $\Sp (C-p)$.
\end{proof}
The next definition introduces the $\mu$-Lipschitz property.
\begin{definition}\label{def_lipsch}
Let $S$ be a subset of a real vector space, $\varphi:S\rightarrow\mathbb{R}$ be an arbitrary function, $D$ be a subset of $S$ and let $\mu$ be a Minkowski functional. We say that the function $\varphi:S\rightarrow\mathbb{R}$ is \textit{$\mu$-Lipschitz on the set $D$ with the constant $L>0$} if for all pairs $\{u,v\}\subset D$ the next inequality holds:
$$ |\varphi(x) - \varphi(y)| \leq L \mu(x-y). $$
The constant $L$ is called \textit{a $\mu$-Lipschitz constant of the function $\varphi$ on the set $C$}. If the constant $L$ is not important in a context, then we simply say that the function \textit{$\varphi$ is $\mu$-Lipschitz on the set $D$.}
We also say that $\varphi$ is \textit{locally $\mu$-Lipschitz on the set $D$} if for every point $x \in D$ there exists $\varepsilon > 0$ such that $\varepsilon (C-p) + x \subset S$ and $\varphi$ is $\mu$-Lipschitz on the set $\varepsilon (C-p) + x$.
\end{definition}

In the next theorem, which is simply obtained from the locally convex case, we establish a $\mu$-Lipschitz property for some convex functions.

\begin{theorem}\label{th_main}
Let $S$ be a non-empty convex subset of a real vector space, $\varphi:S\rightarrow \mathbb{R}$ be a convex function and let $C$ be a convex subset of $S$ that is symmetric with respect to a point $p \in C$ and such that $\varphi|_C$ is bounded above. Let $\mu$ denotes the Minkowski functional of the set $C-p$. Then for all $\varepsilon  \in (0,1)$ and for all pairs \text{$\{u,v\} \subset \varepsilon (C-p)+p$} the next inequality holds:
\begin{equation}|\varphi(u) - \varphi(v)| \leq M\frac{1+\varepsilon}{1-\varepsilon} \mu(u-v),\ \text{ where } M:=  \sup_{x \in C}(\varphi(x)-\varphi(p)),\label{eq:lipsch}\end{equation}
i.e. $\varphi$ is $\mu$-Lipschitz on the set $\varepsilon (C-p) + p$ with the constant $M(1+\varepsilon)(1-\varepsilon)^{-1}$. 
\end{theorem}
For instance, the function $\varphi$ is bounded above on a set $C_A$, which is defined in \eqref{eq:c_a}. 
\begin{proof} 
Let us endow the space $\Sp (C-p)$ with the strongest locally convex topology (see, for example, Edwards \cite[Subsection 1.10.1]{edwards}) and let $\mu$ be the Minkowski functional of the set $C-p$. By Lemma \ref{absorb}, the set $C-p$ is absorbing in $\Sp (C-p)$; hence, the Minkowski functional $\mu$ is a continuous seminorm in $\Sp(C-p)$. Then the set
$U:=\{x \in \Sp (C-p): \mu(x) \leq 1\}$
is a closed convex neighbourhood of zero and $C-p$ is a subset of $U$. Note that $\mu_U = \mu$ (see, e.g., Rudin \cite[Theorem 1.35]{rudin}). Using Theorem 2.2.11 from Z\u alinescu \cite{zalinescu}, we obtain that for every $\varepsilon \in (0,1)$ and for all $\{u,v\}\subset \varepsilon (C-p) + p$ the inequality in \eqref{eq:lipsch} holds.
\end{proof}
Theorem 2.2.11 from Z\u alinescu \cite{zalinescu}, which is cited in Theorem \ref{th_main}, is actually proved in the case of separated locally convex spaces (the assumption of the whole chapter that consists \text{Theorem 2.2.11}), though it doesn't use separability, and thus the theorem can be applied to the strongest locally convex topology. In Counterexample \ref{cntex_epsilon} it is shown that the function $\varphi$ from Theorem \ref{th_main} may not be $\mu$-Lipschitz on the set $B_\mu(p,1)$; hence, a choice of $\varepsilon \in (0,1)$ is necessary. It also can be inferred from Counterexample \ref{cntex_convexity} that Theorem \ref{th_main} is not valid for quasiconvex functions. Counterexample \ref{cntex_yravn} shows that the condition of symmetry of the set $C$ with respect to some point cannot be omitted.
\begin{theorem} 
\label{loc_lip}
Let $S$ be a convex subset of a real vector space such that $\icr S \neq \emptyset$, $\varphi:S\rightarrow\mathbb{R}$ be a convex function and let $C$ be a convex subset of $S$ that is symmetric with respect to a point $p \in C$ and such that $\varphi|_C$ is bounded above. Let $\mu$ be the Minkowski functional of the set $C-p$. Then the function $\varphi$ is locally $\mu$-Lipschitz on $\icr S$.
\end{theorem}
\begin{proof}
Let us fix any $\varepsilon \in (0,1)$. By \text{Theorem \ref{th_main}}, the function $\varphi$ is $\mu$-Lipschitz on $\varepsilon (C-p)+p$. Let $x \in \icr S$. Then there exists $t> 0$ such that $x + t(x-p) \in S$. Let $y \in \varepsilon (C-p) + p$. Since $S$ is convex,
\begin{equation} \frac{t}{1+t} y + \frac{1}{1+t} (x + t(x-p)) \in S,\label{eq:point_behind} \end{equation}
therefore,
$$ \frac{t}{1+t}(\varepsilon (C-p) + p) + \frac{1}{1+t} (x + t(x-p)) = \frac{t\varepsilon}{1+t}(C-p) + x  \subset S.$$
It follows from \eqref{eq:point_behind} that $\varphi$ is bounded above on $\varepsilon t (1+t)^{-1} (C-p) + x$. Indeed,
$$
\varphi\left(\frac{t}{1+t} y + \frac{1}{1+t} (x + t(x-p))\right) \leq \frac{t}{1+t} \sup_{y \in  C}\varphi(y) + \frac{1}{1+t} \varphi(x+t(x-p)) < +\infty.
$$
Note that for any $\alpha > 0$ and $u \in \Sp (C-p)$ the equality $\mu_{\alpha (C-p)}(u) = \alpha^{-1}\mu (u)$ holds. Therefore, applying Theorem \ref{th_main} to the set $\varepsilon t (1+t)^{-1}(C-p) + x$, we see that there exists $\lambda \in (0,1)$ such that $\varphi$ is $\mu$-Lipschitz on $\lambda (C-p) + x$. Thus $\varphi$ is locally $\mu$-Lipschitz.
\end{proof}
\begin{corollary}\label{loc_lip_cor}
Let $S$ be a convex subset of a real vector space such that $\icr S \neq \emptyset$, $\varphi:S\rightarrow\mathbb{R}$ be a convex function and let $C_A$ be a set constructed in \eqref{eq:c_a}. Let $\mu$ denotes the Minkowski functional of the set $C_A-x_0$, where $x_0$ is a point from $C_A$ such that $C_A-x_0 = -(C_A-x_0)$. Then the function $\varphi$ is locally $\mu$-Lipschitz on $\icr S$.
\end{corollary}
\begin{proof}
Since the function $\varphi|_{C_A}$ is bounded above, the statement of the corollary follows immediately  from \text{Theorem \ref{loc_lip}.\qedhere}
\end{proof}
Corollary \ref{loc_lip_cor} is illustrated in Example \ref{example_conv} with Proposition \ref{lipschlemma}.

\section{The Clarke subdifferential of $\mu$-Lipschitz functions}\label{section_clarke}
In this section it is considered an abritrary function defined on a convex set that is symmetric with respect to some point. This function is assumed to be Lipschitz with respect to a Minkowski functional. In \text{Subsection \ref{subsection_gdd}} it is studied a generalized directional derivative of the considered function. In \text{Subsection \ref{subsection_elprop}} we define a Clarke subdifferential that is appropriate to the studied case, and we investigate its properties. In \text{Subsection \ref{subsection_calculus}} it is provided a calculus of the defined subdifferential. The properties and the calculus of the Clarke subdifferential are simply translated from the case of Banach spaces to the case of vector spaces (or to topological vector spaces if continuity properties are used)  using certain type of quotient spaces.

Let $C$ be a convex subset of a real vector space such that $C$ is symmetric with respect to a point $x_0 \in C$. 
Denote
\begin{equation} \mu := \text{the Minkowski functional of the set }C-x_0.\label{eq:muc}
\end{equation}
 and let $\varphi : C \rightarrow \mathbb{R}$ be a $\mu$-Lipschitz function with the constant $L>0$. 
Denote
\begin{equation}\label{eq:x_0}X_0:= \Sp (C-x_0), \ Y := X_0/_{\Ker \mu}. \end{equation}
We endow the space $Y$ with a structure of a normed space. The norm is defined as
\begin{equation}\label{eq:pi}\| \pi(x) \| := \mu(x), \ \text{ where } \pi:X_0 \rightarrow Y \text{ is a quotient map.}\end{equation}
Next, put
\begin{equation}\overline{Y} := \text{ the complement of the space }Y\text{ that is a Banach space.}\label{eq:y_c}\end{equation}
Consider the following commutative diagram:
\begin{equation}\xymatrix{ X_0 \ar[d]_{m} \ar[r]^{\ \ \pi} &Y \ar[ld]^i\\ \overline{Y}}\label{eq:comdiagr}\end{equation}
where $i$ is a linear isometry and $m := i \circ \pi$.
Next, define the function
\begin{equation}\label{eq:psi} \psi (x) := \lim_{n \rightarrow \infty} \varphi(y_n), \text{ where } \{y_n\}_{n=1,2,\ldots} \subset C, \ m(y_n-x_0) \rightarrow x.\end{equation}
Since $\varphi$ is $\mu$-Lipschitz, the function $\psi$ is well-defined on $\cl m (C) $. We see that in the space $\overline{Y}$
$$ \Int \cl m (C) = B(0,1): = \{x \in \overline{Y} : \|x\| < 1\}.$$
Each time when a vector topology in $X_0$ is considered, we supply the set $B_{\mu}(x_0,1)$ with the following topology:
\begin{equation}\label{eq:topology} U \text{ is open in } B_{\mu}(x_0,1) \ \Leftrightarrow \ U - x_0 \text{ is open in the vector topology considered in } X_0.\end{equation}
\subsection{The elementary properties of $\varphi^\circ$}\label{subsection_gdd}
It is natural to define the generalized directional derivative of the function $\varphi$ as
\begin{equation}\label{eq:gdd} \varphi^\circ(x,v) := \limsup_{\substack{\mu(y-x)\rightarrow 0, \\ t\rightarrow 0+}}\frac{\varphi(y+tv) - \varphi(y)}{t},\ \ x \in B_{\mu}(x_0,1), \ v \in X_0.\end{equation}
If the set $C$ is a ball in a Banach space, then such a definition of generalized directional derivative coincides with Clarke's definition; therefore, the notation is not ambiguous.
\begin{lemma}\label{equality_of_gdd}
Let $C$ be a convex subset of a real vector space such that $C$ is symmetric with respect to a point $x_0 \in C$. Let $\mu$ be the Minkowski functional of the set $C-x_0$ and let $\varphi : C \rightarrow \mathbb{R}$ be a $\mu$-Lipschitz function with the constant $L>0$. Let $\psi$ be the function defined in \eqref{eq:psi}. Then the generalized directional derivatives $\psi^\circ$ and $\varphi^\circ$ coincide:
$$\psi^\circ(m(x-x_0),m(v)) = \varphi^\circ(x,v),\ \ x \in B_{\mu}(x_0,1), \ v \in X_0,$$
where $\varphi^\circ(\cdot,\cdot)$ is defined in \eqref{eq:gdd}, $m$ is defined in \eqref{eq:comdiagr} and $X_0$ is defined in \eqref{eq:x_0}.
\end{lemma}
\begin{proof}
Let us fix any $x \in B_{\mu}(x_0,1)$ and $v \in X_0$. There exist a sequence $\{y_n\}_{n=1,2,\ldots}$ in $B_\mu(x_0,1)$ and a sequence $\{t_n\}_{n=1,2,\ldots}$ of positive real numbers such that $\mu(y_n - x) \rightarrow 0$, $t_n \rightarrow 0+$ and
$$ \varphi^\circ(x,v) = \limsup_{n \rightarrow \infty} \frac{\varphi(y_n+t_nv) - \varphi(y_n)}{t_n}.$$
Then
$$ \varphi^\circ(x,v) = \limsup_{n \rightarrow \infty} \frac{\psi(m(y_n-x_0) + t_n m(v)) - \psi(m(y_n-x_0))}{t_n} \leq \psi^\circ(m(x-x_0),m(v)).$$
On the other hand, since $m(X_0)$ is dense in $\overline{Y}$, there exist a sequence $\{u_n\}_{n=1,2,\ldots}$ in $B_\mu(x_0,1)$ and a sequence $\{t_n\}_{n=1,2,\ldots}$ of positive real numbers such that $\|m(u_n-x_0) - m(x-x_0)\| \rightarrow 0$, $t_n \rightarrow 0+$ and
$$\psi^\circ(m(x-x_0), m(v)) = \limsup_{n\rightarrow \infty}\frac{\psi(m(u_n-x_0) + t_n v) - \psi(m(u_n-x_0))}{t_n}.$$
Then
$$\psi^\circ(m(x-x_0),m(v)) = \limsup_{n\rightarrow \infty}\frac{\varphi(u_n + t_n v) - \varphi(u_n)}{t_n} \leq \varphi^\circ(x,v).$$
Thus $\psi^\circ(m(x-x_0),m(v)) = \varphi^\circ(x,v)$.
\end{proof}
The following proposition describes the algebraic properties of $\varphi^\circ$, \text{which are taken from $\psi^\circ$.} In case of a Banach space, see Clarke \cite[Proposition 2.1.1]{clarke_book}.
\begin{proposition}\label{prop_alg}Let $C$ be a convex subset of a real vector space such that $C$ is symmetric with respect to a point $x_0 \in C$. Let $\mu$ be the Minkowski functional of the set $C-x_0$ and let $\varphi : C \rightarrow \mathbb{R}$ be a $\mu$-Lipschitz function with the constant $L>0$. Let $\varphi^\circ$ be the function defined in \eqref{eq:gdd}  and $X_0$ be the space defined in \eqref{eq:x_0}. Then for all $x \in B_{\mu}(x_0,1)$ the following statements hold:
\begin{enumerate}[(i)]
\item The function $\varphi^\circ(x,\cdot)$ is positive homogeneous and subadditive;
\item The function $\varphi^\circ(x,\cdot)$ is $\mu$-Lipschitz on $X_0$ with the constant $L$;
\item For all $v \in X_0$: $\varphi^\circ(x,-v) = (-\varphi)^\circ(x,v)$.
\end{enumerate}
\end{proposition}
\begin{proof}
Let $x \in B_{\mu}(x_0,1)$, $\{v,w\} \subset X_0$, $\alpha > 0$ and let $\psi$ be the function defined in \eqref{eq:psi}.
Then
$$\psi^\circ(m(x-x_0),m(v)+m(w)) \leq \psi^\circ(m(x-x_0),m(v)) + \psi^\circ(m(x-x_0),m(w)),$$
where $m$ is defined in \eqref{eq:comdiagr}. Therefore, by Lemma \ref{equality_of_gdd},
$$ \varphi^\circ(x, v+w) \leq \varphi^\circ(x,v) + \varphi^\circ(x,w).$$
Similarly, since $\psi^\circ(m(x-x_0), \alpha v) = \alpha \psi^\circ(m(x-x_0),v)$, we see that $\varphi^\circ(x,\alpha v) = \alpha \varphi^\circ(x,v)$.
Next, since $\psi^\circ(m(x-x_0),\cdot)$ is Lipschitz,
$$ |\psi^\circ(m(x-x_0),m(v)) - \psi^\circ(m(x-x_0),m(w))| \leq L \mu(v-w);$$
hence, by Lemma \ref{equality_of_gdd}, we obtain that $\varphi^\circ(x,\cdot)$ is $\mu$-Lipschitz on $X_0$ with the constant $L>0$. The last statement follows from Lemma \ref{equality_of_gdd}.
\end{proof}
The next proposition describes the upper semi-continuity property of $\varphi^\circ(\cdot,\cdot)$.
\begin{proposition}\label{upsgdd}
Let $C$ be a convex subset of a real vector space such that $C$ is symmetric with respect to a point $x_0 \in C$. Let $\mu$ be the Minkowski functional of the set $C-x_0$ and let $\varphi : C \rightarrow \mathbb{R}$ be a $\mu$-Lipschitz function with the constant $L>0$. Let $X_0$ be the space defined in \eqref{eq:x_0} and suppose that it is endowed with a vector topology such that $\mu$ is continuous. Then the function $\varphi^\circ(\cdot,\cdot)$ is upper semi-continuous on $B_{\mu}(x_0,1) \times X_0$, where $B_{\mu}(x_0,1)$ is considered with the topology defined in \eqref{eq:topology}.
\end{proposition}
\begin{proof}
According to Clarke \cite[Proposition 2.2.1.]{clarke_book}, the function $\psi^\circ(\cdot,\cdot)$ is upper semi-continuous on $B(0,1)\times \overline{Y}$.
Let $\{(x_\alpha,v_\alpha)\}_{a \in A}$ be an arbitrary net in $B_{\mu}(x_0,1)\times X_0$ converging to $(x,v) \in B_{\mu}(x_0,1) \times X_0$, and let $m$ and $\pi$ be the maps defined in \eqref{eq:comdiagr}. Since $\mu$ is continuous, $\pi$ is also continuous; hence, $m$ is continuous and $(m(x_\alpha-x_0),m(v_\alpha)) \rightarrow (x,v)$. Therefore,
$$ \limsup_{\substack{x_\alpha\rightarrow x, \\ v_\alpha \rightarrow v}} \psi^\circ(m(x_\alpha-x_0),m(v_\alpha)) \leq \psi^\circ(m(x-x_0),m(v)).$$
By Lemma \ref{equality_of_gdd},
$$ \limsup_{\substack{x_\alpha\rightarrow x, \\ v_\alpha \rightarrow v}}\varphi^\circ(x_\alpha,v_\alpha) \leq \varphi^\circ(x,v),$$
i.e. $\varphi^\circ(\cdot,\cdot)$ is upper semi-continuous.
\end{proof}

\subsection{The elementary properties of $\partial^\mu_{C}$}\label{subsection_elprop}
Let $X_0$ be the space defined in \eqref{eq:x_0} and $\overline Y$ be the space defined in \eqref{eq:y_c}. Let $m$ be the map defined in \eqref{eq:comdiagr}.
Consider the map $M : (\overline{Y})^* \rightarrow X_0^\prime$ \text{defined by}
\begin{equation}\label{eq:M}M(\zeta) := \zeta \circ m, \ \ \zeta \in (\overline Y)^*.\end{equation}
In the following lemma we establish a few properties of the map $M$.
\begin{lemma}\label{propmaps}
Let $M$ be the map defined in \eqref{eq:M}. Then the map $M$ is linear and injective. If additionally the space $X_0$ defined in \eqref{eq:x_0} is supplied with a vector topology such that $\mu$ defined in \eqref{eq:muc} is continuous, then $M(\zeta) \in X_0^*$ for every $\zeta \in (\overline Y)^*$ and $M$ is continuous with respect to $\sigma(X_0^*, X_0)$ and $\sigma((\overline Y)^*, Y)$ topologies.
\end{lemma}

\begin{proof}
The linearity is obvious. If $M(\zeta) = 0$, then $\left\langle \zeta, m(x)\right\rangle = 0$ for all $x \in X_0$, i.e. $\left\langle \zeta, y \right\rangle = 0$ for all $y \in m(X_0)$. Since $\zeta$ is continuous and $m(X_0)$ is a dense subspace of $\overline{Y}$, we see that $\zeta = 0$. Therefore, $M$ is injective. Next, suppose that $X_0$ is supplied with a vector topology such that $\mu$ is continuous. Then $m$ is continuous and therefore $\zeta \circ m \in X_0^*$ for any $\zeta \in (\overline Y)^*$. If $\{\zeta_\alpha\}_{\alpha \in A}$ is a net in $(\overline Y)^*$ such that for all $x \in \overline Y$
$$ \left\langle \zeta_\alpha, x \right\rangle \rightarrow 0, $$
then a fortiori for all $x \in X_0$
$$\left \langle \zeta_\alpha, m(x) \right\rangle \rightarrow 0, $$
i.e. $M(\zeta_\alpha) \rightarrow 0$ in $\sigma(X_0^*, X_0)$ topology. Thus $M$ is continuous.
\end{proof}

Let $C$ be a convex subset of a real vector space such that $C$ is symmetric with respect to a point $x_0 \in C$. Let $\mu$ be the Minkowski functional of the set $C-x_0$ and let $\varphi : C \rightarrow \mathbb{R}$ be a $\mu$-Lipschitz function. Consider the Clarke subdifferential of the function $\psi$ defined in \eqref{eq:psi}:
$$ \partial_{C}\psi(x) = \{ \zeta \in (\overline{Y})^* : \psi^\circ(x,v) \geq \left\langle \zeta,v\right\rangle \text{ for all } v \in \overline{Y}\},\ \ x \in B(0,1). $$
Applying the map $M$ defined in \eqref{eq:M} to $\partial_C\psi(x)$ and using Lemma \ref{equality_of_gdd}, we see that
$$M (\partial_{C}\psi(m(x-x_0))) = \{\zeta \in X_0^\prime : \varphi^\circ(x,v) \geq \left\langle \zeta, v \right\rangle \text{ for all } v \in X_0  \},\ \ x \in B_{\mu}(x_0,1); $$
therefore, it is natural to say that $M(\partial_{C}\psi(m(x-x_0)))$ is the Clarke subdifferential of the function $\varphi$ at the point $x$, and we denote
\begin{equation} \partial^\mu_{C}\varphi(x):  = M(\partial_{C}\psi(m(x-x_0))),\  \ x\in B_{\mu}(x_0,1).\label{eq:clarke_df} \end{equation}
Since the definition of $\varphi^\circ$ in \eqref{eq:gdd} depends on the Minkowski functional $\mu$, it is included in the formulae \eqref{eq:clarke_df}.
It follows from \eqref{eq:clarke_df} and Lemma \ref{equality_of_gdd} that
$$ \varphi^\circ(x,v) = \max\{\left\langle \zeta, v \right\rangle : \zeta \in \partial \varphi (x)\}.$$
The next proposition describes some elementary properties of $\partial^\mu_{C}\varphi(\cdot)$.
\begin{proposition}\label{subd_gen_properties}
Let $\partial^\mu_C\varphi(\cdot):B_{\mu}(x_0,1) \rightrightarrows X_0^\prime$ be the multivalued map defined in \eqref{eq:clarke_df}. Then $\partial^\mu_C\varphi(\cdot)$ has non-empty convex values. If additionally the space $X_0$ defined in \eqref{eq:x_0} is supplied with a vector topology such that $\mu$ is continuous, then $\partial^\mu_{C}\varphi(\cdot)$ takes compact values in the space $X_0^*$ with respect to $\sigma(X_0^*,X_0)$ topology.
\end{proposition}
\begin{proof}
According to Clarke \cite[Proposition 2.1.2.]{clarke_book}, the Clarke subdifferential of the function $\psi$ defined in \eqref{eq:psi} has non-empty compact convex values with respect to $\sigma((\overline{Y})^*,\overline{Y})$ topology. Let $M$ be the map defined in \eqref{eq:M}. It follows from Lemma \ref{propmaps} and equality in \eqref{eq:clarke_df} that $\partial^\mu_{C}\varphi(\cdot)$ has non-empty convex values. Since every $\zeta \in \partial^\mu_{C}\varphi(x)$ is bounded above by $\mu$, we see that $\partial^\mu_{C}\varphi(x) \subset X_0^*$ whenever the space $X_0$ is supplied with a vector topology such that $\mu$ is continuous.
As well if $\mu$ is continuous, then, by Lemma \ref{propmaps}, the map $M$ is continuous; hence, the multivalued map $\partial^\mu_{C}\varphi(\cdot)$ has compact values in $X_0^*$ with respect to $\sigma(X_0^*,X_0)$ topology.
\end{proof}

\begin{proposition}\label{usc}
Let $\partial^\mu_{C} \varphi(\cdot) : B_{\mu} (x_0,1) \rightrightarrows X_0^\prime$ be the multivalued map defined in \eqref{eq:clarke_df} and let the space $X_0$ defined in \eqref{eq:x_0} be supplied with a vector topology such that $\mu$ is continuous. Next, let the set $B_{\mu}(x_0,1)$ be supplied with the topology defined in \eqref{eq:topology}. Then the multivalued map $\partial^\mu_C\varphi(\cdot)$ is upper semi-continuous with respect to the topology in $B_{\mu}(x_0,1)$ and \text{$\sigma(X_0^*, X_0)$ topology in $X_0^*$.}
\end{proposition}
\begin{proof}
It follows from Clarke \cite[Proposition 2.1.5.]{clarke_book} that the Clarke subdifferential $\partial_C\psi(\cdot)$ is u.s.c. on $B(0,1).$
Let the space $X_0$ be supplied with a vector topology such that $\mu$ is continuous and let $M$ be the map defined in \eqref{eq:M}.
Let $Q \subset X_0^*$ be a closed subset in $\sigma(X_0^*, X_0)$. By Lemma \ref{propmaps}, the map $M$ is continuous; hence, $M^{-1}(Q)$ is closed in $\sigma((\overline{Y})^*, \overline{Y})$. Since $\partial_{C} \psi(\cdot)$ is u.s.c., the set
$$ U := \{ y \in B(0,1) : \partial_{C} \psi (y) \cap M^{-1}(Q) \neq \emptyset\} $$
is closed. Since $m$ is continuous, we obtain that $m^{-1} (U)$ is closed; hence, by definition of the topology in $B_\mu(x_0,1)$, the set $m^{-1}(U) +x_0$ is closed in $B_\mu(x_0,1)$ and
$$ m^{-1}(U)+x_0 = \{ x \in B_\mu(x_0,1) : \partial_{C}\psi(m(x-x_0)) \cap M^{-1}(Q) \neq \emptyset\};$$
therefore, it is enough to show that
$$ \partial_{C}\psi(m(x-x_0)) \cap M^{-1}(Q) \neq \emptyset \ \Leftrightarrow  \ \partial^\mu_C\varphi(x) \cap Q \neq \emptyset.$$
Applying the map $M$, we see that
$$ M(\partial_{C}\psi(m(x-x_0)) \cap M^{-1}(Q)) \subset \partial^\mu_C\varphi(x) \cap M M^{-1}(Q) = \partial^\mu_C\varphi(x) \cap Q,$$
hence if $\partial_{C}\psi(m(x-x_0))\cap M^{-1}(Q) \neq \emptyset$, then $\partial^\mu_C\varphi(x) \cap Q \neq \emptyset$. On the other hand, let $\partial^\mu_C\varphi(x) \cap Q \neq \emptyset$. Since $\partial^\mu_C\varphi(x) = M(\partial_C(\psi(m(x-x_0)))$, we see that $Q$ intersects with the image of the map $M$, therefore
$$\emptyset \neq M^{-1}(\partial^\mu_C\varphi(x) \cap Q) = M^{-1}M(\partial_C\psi(m(x-x_0)))\cap M^{-1}(Q) \subset \partial_C\psi(m(x-x_0)) \cap M^{-1}(Q).$$
Thus $\partial^\mu_C\varphi(\cdot)$ is u.s.c.
\end{proof}
\begin{proposition}(Fermat's rule)\label{fermat}
Let $C$ be a convex set that is symmetric with respect to a point $x_0 \in C$, $\mu$ be the Minkowski functional of the set $C-x_0$ and let $\varphi : C \rightarrow \mathbb{R}$ be a $\mu$-Lipschitz function with the constant $L>0$. Suppose that there exists $\varepsilon > 0$ and $u \in C$ such that $\varepsilon(C-x_0) + u \subset C$ and the restriction of $\varphi$ to $\varepsilon(C-x_0) + u$ attains a minimum or a maximum at the point $u$. Then $\varphi^\circ(u,v)\geq 0$ for all $v \in X_0$ and $\partial^\mu_{C} \varphi(u) \ni 0$, where $X_0$ is defined in \eqref{eq:x_0}.
\end{proposition}
\begin{proof}
If $\varphi_{|\varepsilon(C-x_0) + u}$ attains a minimum or a maximum at $u$, then $\psi_{|m(\varepsilon(C-x_0)) + m(u-x_0)}$ defined in \eqref{eq:psi} attains a minimum or a maximum at $m(u-x_0)$, where $m$ is defined in \eqref{eq:comdiagr}. Note that $m(u-x_0) \in \Int m(\varepsilon(C-x_0))$, hence $\psi^\circ(m(u-x_0),m(v)) \geq 0$ for all $v \in X_0$ and $\partial_{C}\psi(m(u-x_0)) \ni 0$. By Lemma \ref{equality_of_gdd}, $\varphi^\circ(u,v) \geq 0$ for all $v \in X_0$ and since $\partial^\mu_C\varphi(u) = M(\partial_{C}\psi(m(u-x_0)))$, we see that $\partial^\mu_C\varphi(u) \ni 0$.
\end{proof}

\begin{theorem}(Lebourg mean-value theorem)\label{lebourg}
Let $C$ be a convex set that is symmetric with respect to a point $x_0 \in C$, $\mu$ be the Minkowski functional of the set $C-x_0$ and let $\varphi : C \rightarrow \mathbb{R}$ be a $\mu$-Lipschitz function with the constant $L>0$. If $\{y,x\} \subset B_{\mu}(x_0,1)$ are distinct points, then there exists a point $z$ in the open line segment between $x$ and $y$ such that
$$ \varphi(y) - \varphi(x) \in \{ \left\langle \zeta, y-x \right\rangle : \zeta \in \partial^\mu_C\varphi(z)\} .$$
\end{theorem}
\begin{proof}
Let $\{x,y\} \subset B_\mu(x_0,1)$. Let $m$ be the map defined in \eqref{eq:comdiagr} and let $\psi$ be the function defined in \eqref{eq:psi}. Consider two cases: \ref{case1}) $y-x \notin \Ker \mu$ and \ref{case2}) $y-x \in \Ker \mu$.
\begin{enumerate}[a)]
\item\label{case1} If $y - x \notin \Ker \mu$ then $m(x-x_0) \neq m(y-x_0)$; hence, according to the classic Lebourg mean value theorem (see, e.g., Clarke \cite[Theorem 2.3.7.]{clarke_book}) applied to the function $\psi$, there exists a point $m(z-x_0)$ in the open line segment between $m(x-x_0)$ and $m(y-x_0)$ such that
$$ \psi(m(y-x_0)) - \psi(m(x-x_0)) \in\{ \left\langle \zeta, m(y-x) \right\rangle : \zeta \in \partial_{C}\psi(m(z-x_0))\}.$$
Since $\left\langle \zeta, m(y-x) \right\rangle  = \left\langle M(\zeta), y-x \right\rangle$, we see that
$$ \varphi(y) - \varphi(x)\in \{ \left\langle \zeta, y-x \right\rangle : \zeta \in \partial^\mu_C\varphi(z)\}.$$

\item\label{case2} Note that if $\zeta \in \partial^\mu_{C}\varphi(x)$ for $x \in B_{\mu}(x_0,1)$, then $\mu(v) \geq \left\langle \zeta, v \right\rangle$ for all $v \in X_0$ and hence $\Ker \zeta \supset \Ker \mu$. Thus if $y-x \in \Ker \mu$, we obtain that
$$ \varphi(y) - \varphi(x) = 0 \text{ and } \varphi(y) - \varphi(x) \in\{ \left\langle \zeta, y-x \right\rangle : \zeta \in \partial^\mu_C\varphi(z)\},$$
where $z$ is an arbitrary point in the line segment between $x$ and $y$.\qedhere
\end{enumerate}
\end{proof}
Theorem \ref{lebourg} is illustrated in Proposition \ref{ex_lebourg}.

\subsection{Calculus of $\partial^\mu_C$}\label{subsection_calculus}
In this subsection we consider chain, sum and multiple rules for $\partial^\mu_C$.
Let $\varphi : C \rightarrow \mathbb{R}$ be a function defined on a convex set $C$ that is symmetric with respect to a point $x_0 \in C$. Let $\mu$ be the Minkowski functional of the set $C-x_0$ and put $X_0:=\Sp(C-x_0)$ 
Following Clarke \cite{clarke_grad}, \cite{clarke_book}, we say that the function $\varphi$ is \textit{regular} at the point $x \in B_{\mu}(x_0,1)$ if the following statements hold:
\begin{enumerate}
\item The directional derivative $\varphi^\prime(x,v) := \lim\limits_{t \rightarrow 0+} {(\varphi(x+tv) - \varphi(x))}{t^{-1}}$ exists for all $v \in X_0$;
\item The equality $\varphi^\prime(x,v) = \varphi^\circ(x,v)$ holds for all $v \in X_0$.
\end{enumerate}
\begin{theorem}(Chain rule II)\label{prop_chain_rule2}
Let $C$ be a convex set that is symmetric with respect to a point $x_0 \in C$, $h:C\rightarrow \mathbb{R}^n$ be a function such that each component $h_i$ is $\mu$-Lipschitz on $C$ and let $g:\mathbb{R}^n\rightarrow\mathbb{R}$ be a Lipschitz function.
Denote $\varphi: = g\circ h$ and let $X_0:=\Sp(C-x_0)$ be supplied with a vector topology such that $\mu$ is continuous. Then for all $x \in B_{\mu}(x_0,1)$ the next inclusion holds:
\begin{equation} \partial^\mu_C\varphi(x) \subset \overline{\co} \left\{\sum_{i=1}^n \alpha_i \zeta_i : \zeta_i \in \partial^\mu_Ch_i(x), \, \alpha \in \partial_Cg(h(x))\right\},\label{eq:chainrule2}\end{equation}
where $\overline{\co}$ is a closed convex hull operator considered in the topology $\sigma(X_0^*,X_0)$.
\end{theorem}
\begin{proof}
Consider the following function
$$\hat{h} (x) := \lim_{n\rightarrow \infty}h(y_n), \ m(y_n-x_0)\rightarrow x, \{y_n\}_{n=1,2,\ldots}\subset C,$$
where $m$ is defined in \eqref{eq:comdiagr}.
The function $\hat{h}$ is well-defined since each component $h_i$ is $\mu$-Lipschitz. Then for every $x \in B_\mu(x_0,1)$ the function $\psi$ defined in \eqref{eq:psi} satisfies the following equality: 
$$\psi(x) = \lim\limits_{n\rightarrow \infty}g\circ h(y_n) = g(\lim\limits_{n\rightarrow\infty}h(y_n)) = g\circ \hat{h}(x),$$ where $\{y_n\}_{n=1,2,\ldots}$ is an arbitrary sequence from $C$ such that $m(y_n-x_0) \rightarrow x$.
Denote
\begin{equation}
 A_x:= \left\{\sum_{i=1}^n \alpha_i \zeta_i : \zeta_i \in \partial_C\hat{h}_i(m(x-x_0)), \, \alpha \in \partial_Cg(\hat{h}(m(x-x_0)))\right\}.\label{eq:a_x}
\end{equation}
Using results from a Banach space (see Clarke \cite[Theorem 2.3.9.]{clarke_book}), we obtain that \text{$ \partial_C\psi(m(x-x_0)) \subset \overline{co} A_x $}, hence 
$$\partial^\mu_C\varphi(x) = M(\partial_C\psi(m(x-x_0))) \subset M\overline{co}A_x \subset \overline{co}M(A_x).$$
Since $M(\partial_C\hat{h}_i(m(x-x_0)) = \partial^\mu_Ch_i(x)$ and $\hat{h}(m(x-x_0)) = h(x)$, we obtain that
$$ \partial^\mu_C\varphi(x) \subset \overline{\co} \left\{\sum_{i=1}^n \alpha_i \zeta_i : \zeta_i \in \partial^\mu_Ch_i(x), \, \alpha \in \partial_Cg(h(x))\right\}.\qedhere$$
\end{proof}
In Example \ref{ex_chain_rule2} we provide a use of Theorem \ref{prop_chain_rule2}.
\begin{corollary}\label{cor_chain_rule1}
Under the assumptions of Theorem \ref{prop_chain_rule2}, the equality holds in \eqref{eq:chainrule2} at the point $x \in B_{\mu}(x_0,1)$ if additionally the following assumptions hold:
\begin{enumerate}[(i)]
\item The function $g$ is regular at the point $h(x)$ and every function $h_i$ is regular at the point $x$;\label{cr2as1}
\item Every element $\alpha \in \partial_Cg(h(x))$ is non-negative;\label{cr2as2}
\item The space $X_0$ is endowed with a complete Hausdorff locally convex topology such that $\mu$ is continuous.\label{cr2as3}
\end{enumerate}
\end{corollary}
\begin{proof}
Let $A_x$ be the set defined in $\eqref{eq:a_x}$, $M$ be the map defined in \eqref{eq:M} and $\overline{Y}$ be the space defined in \eqref{eq:y_c}. 
The set $A_x$ is compact in $\sigma{((\overline{Y})^*, \overline{Y})}$, and it follows from assumption \eqref{cr2as3} that the set $\overline{\co}(A_x)$ is compact in $\sigma{((\overline{Y})^*, \overline{Y})}$ (see the remark in Edwards \cite[p. 231]{edwards}). Therefore, the set $M(\overline{\co}{A_x})$ is compact and convex in $\sigma{(X_0^*, X_0)}$, thus $M(\overline \co A_x) = \overline \co M(A_x)$.\footnote{Here is an extended explanation. It is known from the general topology that the map $M$ is continuous if and only if for any set $A$ the inclusion $M(\cl A) \subset \cl M(A)$ holds. Since the map $M$ is linear, we see that $M(\co A) = \co M (A)$; therefore, for any set $A$ we obtain that $M(\cl \co A) \subset \cl M(\co A) = \cl \co M(A)$, i.e. $M(\overline{\co}A) \subset \overline{\co}M(A)$. Since $\overline\co A \supset A$, we see that $M(\overline \co A) \supset M(A)$, hence $\overline\co M(\overline \co A) \supset \overline \co M(A)$. If the set $M(\overline \co A)$ is compact, then the outside operator $\overline \co$ can be omitted. Here comes out the remark from Edwards: in a complete Hausdorff locally convex topology the operator $\overline \co$ preserves a set to be compact, i.e. if $A$ is compact, then $\overline\co A$ is compact and convex, and hence $M(\overline \co A)$ is compact and convex. }
Note that the functions $\hat h_i$ are regular at the point $m(x-x_0)$ since the functions $h_i$ are regular at the point $x$. According to Clarke \cite[Theorem 2.3.9.]{clarke_book}, the next equality holds:
$$\partial_C\psi(m(x)) = \overline{\co}A_x;$$
thus, applying $M$ to both sides, we obtain that
$$\partial^\mu_C\varphi(x) = M(\overline \co A_x) = \overline \co M(A_x) = \overline \co \left\{\sum_{i=1}^n \alpha_i \zeta_i : \zeta_i \in \partial^\mu_Ch_i(x), \, \alpha \in \partial_Cg(h(x))\right\}.\qedhere$$
\end{proof}

\begin{theorem}(Sum rule) Let $Q$ and $C$ be convex subsets of a real vector space such that $Q \cap C\neq \emptyset$, $\Sp(C-x_0) = \Sp(Q-x_0)$ and such that both $Q$ and $C$ are symmetric with respect to a point $x_0 \in Q\cap C$. Let $\mu$ be the Minkowski functional of the set $C-x_0$, $\nu$ be the Minkowski functional of the set $Q-x_0$. Next, let $\varphi:C\rightarrow \mathbb{R}$ be a $\mu$-Lipschitz function and $f:Q\rightarrow\mathbb{R}$ be a $\nu$-Lipschitz function. Then for all $x \in Q \cap C$ the next inclusion holds: \label{sumrule}
\begin{equation}
\partial^{\mu+\nu}_{C}(\varphi+f)(x) \subset \partial^\mu_C\varphi(x) + \partial^\nu_{C}f(x).\label{eq:sum_incl}
\end{equation}
\end{theorem}
The proof is similar to Clarke \cite[Proposition 2.3.3]{clarke_book} except the fact that the generalized directional derivatives $(\varphi + f)^\circ$, $\varphi^\circ$ and $f^\circ$ are taken with different filters.

\begin{proof}
Let $x \in Q\cap C$ and let the space $X_0:=\Sp(C-x_0)$ be supplied with the strongest locally convex topology. Then $\mu$ and $\nu$ are continuous; hence, by  Proposition \ref{subd_gen_properties}, the Clarke subdifferential $\partial^{\mu+\nu}_C(\varphi+f)(x)$ and the set $\partial^{\mu}_C\varphi(x)+ \partial^\nu_Cf(x)$ are convex and compact in $\sigma(X_0^*,X_0)$ topology. Therefore, it is enough to show that for any $v \in X_0$ the next inequality holds:
$$(\varphi+f)^\circ(x,v) \leq \varphi^\circ(x,v) + f^\circ(x,v).$$
Let $v \in X_0$ be an arbitrary vector. Then
$$ (\varphi + f)^\circ(x,v) = \limsup_{\substack{(\mu+\nu)(y-x)\rightarrow 0, \\ t\rightarrow 0+}}\frac{(\varphi+f)(y+tv)-(\varphi+f)(y)}{t} \leq$$$$\leq \limsup_{\substack{(\mu+\nu)(y-x)\rightarrow 0, \\ t\rightarrow 0+}}\frac{\varphi(y+tv)-\varphi(y)}{t} + \limsup_{\substack{(\mu+\nu)(y-x)\rightarrow 0, \\ t\rightarrow 0+}}\frac{f(y+tv)-f(y)}{t} \leq$$$$\leq \varphi^\circ(x,v) + f^\circ(x,v),$$
thus the inclusion holds in \eqref{eq:sum_incl}.
\end{proof}
Theorem \ref{sumrule} is illustrated in Example \ref{ex_sum_rule}.
\begin{corollary}\label{corsumrule}
Under the assumptions of Theorem \ref{sumrule}, if additionally both the functions $\varphi$ and $f$ are regular at the point $y \in Q\cap C$, then $$\partial^{\mu+\nu}_{C}(\varphi+f)(y) = \partial^\mu_C\varphi(y) + \partial^\nu_{C}f(y).$$
\end{corollary}
\begin{proof}
Suppose that $\varphi$ and $f$ are regular at the point $y \in Q \cap C$. Then
$$(\varphi+f)^\circ(y,v) \geq (\varphi+f)^\prime(y,v) = \varphi^\prime(y,v) + f^\prime(y,v) = \varphi^\circ(y,v) + f^\circ(y,v),$$
hence $\partial^{\mu+\nu}_{C}(\varphi+f)(y) = \partial^\mu_C\varphi(y) + \partial^\nu_{C}f(y)$.
\end{proof}
\begin{theorem}\label{multiplerule}
(Multiple rule) Let $Q$ and $C$ be convex subsets of a real vector space such that $Q \cap C\neq \emptyset$, $\Sp(C-x_0) = \Sp(Q-x_0)$ and such that both $Q$ and $C$ are symmetric with respect to a point $x_0 \in Q\cap C$. Let $\mu$ be the Minkowski functional of the set $C-x_0$, $\nu$ be the Minkowski functional of the set $Q-x_0$. Next, let $\varphi:C\rightarrow \mathbb{R}$ be a $\mu$-Lipschitz function and $f:Q\rightarrow\mathbb{R}$ be a $\nu$-Lipschitz function. Then for all $x \in Q \cap C$ the next inclusion holds:
$$\partial^{\mu+\nu}_C(\varphi f)(x) \subset f(x)\partial^\mu_C\varphi(x) + \varphi(x)\partial^\nu_Cf(x).$$
\end{theorem}
\begin{proof}
As in Proposition \ref{sumrule}, let $x \in Q\cap C$ and let the space $X_0:=\Sp(C-x_0)$ be supplied with the strongest locally convex topology. Then $\mu$ and $\nu$ are continuous; hence, by  Proposition \ref{subd_gen_properties}, the Clarke subdifferential $\partial^{\mu+\nu}_C(\varphi f)(x)$ and the set $f(x)\partial^{\mu}_C\varphi(x)+ \varphi(x)\partial^\nu_Cf(x)$ are convex and compact in $\sigma(X_0^*,X_0)$ topology. Therefore, it is enough to show that for any $v \in X_0$ the next inequality holds:
$$(\varphi f)^\circ(x,v) \leq f(x)\varphi^\circ(x,v) + \varphi(x)f^\circ(x,v).$$
Let $v \in X_0$ be an arbitrary vector. Then
$$(\varphi f)^\circ (x,v) = \limsup_{\substack{(\mu+\nu)(y-x)\rightarrow 0, \\ t\rightarrow 0+}}\frac{\varphi(y+tv)f(y+tv) - \varphi(y)f(y)}{t} \leq $$$$\leq \limsup_{\substack{(\mu+\nu)(y-x)\rightarrow 0, \\ t\rightarrow 0+}}f(y+tv)\frac{\varphi(y+tv)-\varphi(y)}{t} + \limsup_{\substack{(\mu+\nu)(y-x)\rightarrow 0, \\ t\rightarrow 0+}}\varphi(y)\frac{f(y+tv)-f(y)}{t} \leq$$ $$\leq f(x)\varphi^\circ(x,v) + \varphi(x)f^\circ(x,v), $$
thus
$$\partial^{\mu+\nu}_C(\varphi f)(x) \subset f(x)\partial^\mu_C\varphi(x) + \varphi(x)\partial^\nu_Cf(x).\qedhere$$
\end{proof}
Theorem \ref{multiplerule} is illustrated in Example \ref{ex_mul_rule}.
\begin{corollary}\label{cor_multiplerule}
Under the assumptions of Theorem \ref{multiplerule}, let additionally the following assumptions hold:
\begin{enumerate}[(i)]
\item The equality $\mu=\nu$ holds;
\item The functions $\varphi$ and $f$ are regular and non-negative at the point $x\in B_{\mu}(x_0,1).$ 
\end{enumerate}
Then $\partial^{\mu}_C(\varphi f)(x) = f(x)\partial^\mu_C\varphi(x) + \varphi(x)^\mu\partial_Cf(x).$
\end{corollary}
\begin{proof}
As in Clarke \cite[Proposition 2.3.13.]{clarke_book}, we define the functions
$$ g:\mathbb{R}^2 \rightarrow \mathbb{R}, \ g(u_1,u_2) = u_1 u_2, \ \ (u_1,u_2) \in \mathbb{R}^2$$
$$ h:B_\mu(x_0,1) \rightarrow \mathbb{R}, \ h(y) = (\varphi(y),f(y)),\ \ x \in B_\mu(x_0,1).$$
Then for all $y \in B_\mu(x_0,1)$ the equality $\varphi(y) f(y) = g \circ h (y)$ holds. Thus, using Corollary \ref{cor_chain_rule1}, we obtain the desired equality.
\end{proof}

For the next theorem we endow the space $X_0$ defined in \eqref{eq:x_0} with a Hausdorff locally convex topology such that $\mu$ is continuous (see the beginning of the section). We also consider a Hausdorff locally convex space $E$, an open convex set $\mathcal{O}\subset E$ and a function \begin{equation}g:\mathcal{O}\rightarrow X_0\label{eq:g}\end{equation} such that the following assumptions hold:
\begin{enumerate}[(i)]
\item The function $g$ is G\^ateaux differentiable on the set $\mathcal{O}$. Moreover, we demand that for all $v \in E$ the function 
$$\mathcal{O}\ni x \mapsto Dg(x)v \in X_0 $$\label{as1}
is continuous, where $Dg(x)$ is the G\^ateaux derivative of the function $g$ at the point $x$;
\item The G\^ateaux derivative $E \ni v \rightarrow Dg(x)v \in X_0$ is continuous for all $x \in \mathcal{O}$\footnote{There are slightly different definitions of the G\^ateaux derivative. We use the definition given in \text{Bogachev et al. \cite{bogachev}}, where the G\^ateaux derivative is a sequantially continuous linear map. If $\zeta \in E^*$, then we need the inclusion $\zeta \circ Dg(x) \in E^*$ to be held; thus, the assumption is necessary.};
\item There exists a point $x \in \mathcal{O}$ such that $g(x) \in B_{\mu}(x_0,1)$;\label{as2}
\item There exists a continuous seminorm $p$ in $E$ such that for all $\{u,w\}\subset \mathcal O$ the next inequality holds:\label{as3}
$$\mu(g(u) - g(w)) \leq p(u-w) \text{ for all }\{u,w\}\subset \mathcal{O}. $$
\end{enumerate}

\begin{theorem}\label{prop_chain_rule1}
(Chain rule I) Let $C$ be a convex set that is symmetric with respect to a point $x_0 \in C$, $\mu$ be the Minkowski functional of the set $C-x_0$ and let $\varphi : C \rightarrow \mathbb{R}$ be a $\mu$-Lipschitz function with the constant $L>0$. Let $g$ be the function defined in \eqref{eq:g}.
Then for every point $x \in \mathcal{O}$ such that $g(x) \in B_{\mu}(x_0,1)$ the next inclusion holds:
\begin{equation}\partial^p_C(\varphi\circ g)(x) \subset \{\zeta\circ Dg(x) : \zeta \in \partial^\mu_C\varphi (g(x)) \}.\label{eq:chainruleone}\end{equation}
\end{theorem}
In case of a Banach space, see Clarke \cite[Theorem 2.3.10.]{clarke_book}.
\begin{proof}
Let us choose $x \in \mathcal O$ such that $g(x) \in B_{\mu}(x_0,1)$. It follows from assumptions \eqref{as2} and \eqref{as3} of the definition of the function $g$ that there exists an open neighbourhood $\mathcal{O}(x)$ of the point $x$ such that $g(\mathcal{O}(x))\subset B_{\mu}(x_0,1)$. The function $\varphi\circ g(\cdot)$ is $p$-Lipschitz on the set $\mathcal{O}(x)$. Indeed, let $\{u,w\}\subset\mathcal{O}(x)$, then
$$|\varphi \circ g(u) - \varphi \circ g(w)| \leq L\mu(g(u) - g(w)) \leq Lp(u-w).$$
It is enough to check the inequality
$$(\varphi\circ g)^\circ(x,v) \leq \varphi^\circ(g(x),Dg(x)v),$$
since both sets in \eqref{eq:chainruleone} are convex and compact in $\sigma(E^*,E)$ topology. Let $\{y,y+tv\}\subset \mathcal{O}(x)$, $t>0$, $v \in E\setminus\{0\}$. By the Lebourg mean value theorem, there exists a point $z_{y,t} \in (g(y),g(y+tv))$ such that
$$\varphi(g(y+tv)) - \varphi(g(y)) = \left\langle \zeta, g(y+tv)-g(y) \right\rangle,$$
where $\zeta \in \partial^\mu_C\varphi (g(z_{y,t}))$ (if $g(y) = g(y+tv)$, then we put $z_{y,t}:=g(y)$). By the generalized mean value theorem (see, e.g., Bogachev et al. \cite[Theorem 12.2.6]{bogachev}), there exists a point $u_{y,t} \in (y,y+tv)$ such that
$$\frac{g(y+tv)-g(y)}{t} = Dg(u_{y,t})v,$$
hence
$$\frac{\varphi(g(y+tv)) - \varphi(g(y))}{t} = \left\langle \zeta, Dg(u_{y,t})v \right\rangle \leq \varphi^\circ(g(x),Dg(u_{y,t})v).$$
Note that if $y \rightarrow x$ and $t\rightarrow 0+$, then $g(y) \rightarrow g(x)$ and $g(y+tv) \rightarrow g(x)$ due to \eqref{as3}. Since the the space $E$ is locally convex, $z_{y,t}\rightarrow g(x)$ and $u_{y,t}\rightarrow x$ as $y \rightarrow x$ and $t \rightarrow 0+$.
By Proposition \ref{upsgdd}, the function $\varphi^\circ(\cdot,\cdot)$ is upper semi-continuous, therefore
$$(\varphi\circ g)^\circ(x,v) \leq \varphi^\circ(x,Dg(x)v),$$
thus 
$$\partial^p_C(\varphi\circ g)(x) \subset \{\zeta\circ Dg(x) : \zeta \in \partial^\mu_C\varphi(g(x)) \}.\qedhere$$
\end{proof}
Theorem \ref{prop_chain_rule1} is illustrated in Example \ref{ex_chain_rule1}.
\begin{corollary}
Under the assumptions of Theorem \ref{prop_chain_rule1}, if $\varphi$ is regular at the point $g(x)$, then the equality holds in \eqref{eq:chainruleone} at the point $x$.
\end{corollary}
\begin{proof}
Let $\varphi$ be regular at the point $g(x)$ and let $v \in X_0$ be an arbitrary vector. Note that
$$ \left|\frac{\varphi(g(x)+tDg(x)v) - \varphi(g(x))}{t} - \frac{\varphi(g(x+tv)) - \varphi(g(x))}{t}\right| \leq $$$$\leq L\mu\left(\frac{g(x+tv)-g(x)}{t}-Dg(x)v\right) \rightarrow 0, \ t\rightarrow 0+,$$
hence
$$\varphi^\prime(g(x),Dg(x)v) = \lim_{t \rightarrow 0+}\frac{\varphi(g(x)+tDg(x)v) - \varphi(g(x))}{t} =$$$$= \lim_{t\rightarrow 0+}\frac{\varphi(g(x+tv)) - \varphi(g(x))}{t} = (\varphi\circ g)^\prime(x,v) \leq (\varphi\circ g)^\circ(x,v),$$
thus 
$$\partial^p_C(\varphi\circ g)(x) = \{\zeta\circ Dg(x) : \zeta \in \partial^\mu_C\varphi(g(x)) \}. \qedhere$$
\end{proof}

\subsection{Other properties}
The propositions in this subsection are treated the same way as in Clarke \cite{clarke_book} with the same conditions.
Let $C$ be a convex subset of a real vector space such that $C$ is symmetric with respect to a point $x_0 \in C$ and let $\mu$ be the Minkowski functional of the set $C-x_0$. Let $\{f_i : C \rightarrow \mathbb{R}\}_{i=1,2,\ldots,n}$ be a collection of $\mu$-Lipschitz functions. We denote
\begin{equation}
 \varphi (x) := \max\{f_1 (x), \ldots, f_n(x)\}, \ \ x \in B_{\mu}(x_0,1) \label{eq:max}
\end{equation}
and put
$$I(x) := \text{ the set of all indexes } i \in \{1,\ldots, n\} \text{ such that }\varphi(x) = f_i(x).$$
\begin{proposition}Let $\varphi$ be the function defined in \eqref{eq:max}. Then the next inclusion holds for all $x \in B_{\mu}(x_0,1)$:
\begin{equation}
\partial^\mu_C \varphi (x) \subset \co \{\partial^\mu_C f_i(x) : i \in I(x)\}.\label{eq:pointwise}
\end{equation}
If additionally all the functions $f_i$ are regular at the point $x$ for each $i \in I(x)$, then the equality holds in \eqref{eq:pointwise}.
\end{proposition}
\begin{proof}
Let $\overline{Y}$ be the space defined in \eqref{eq:y_c}, $m$ be the map defined in \eqref{eq:comdiagr}, $\psi$ be the function defined in \eqref{eq:psi}  and M be the map defined in \eqref{eq:M}. As at the beginning of the section, we consider the duplicates of the functions $f_i$ in the space $\overline{Y}$:
$$ \hat{f}_i(x) := \lim_{n\rightarrow \infty} f_i(y_n), \ m(y_n-x_0) \rightarrow x, \ \{y_n\}_{n=1,2,\ldots}\subset C, \ i \in \{1,\ldots,n\}.$$
According to Clarke \cite[Proposition 2.3.12.]{clarke_book},  for all $x \in B_{\mu}(x_0,1)$
$$\partial \psi (m(x-x_0)) \subset \co \{\partial \hat{f}_i(m(x-x_0)) : i \in I(x)\}.$$
Applying the map $M$ to both sides, we obtain that the inclusion in \eqref{eq:pointwise} holds. Similarly if $f_i$ are regular at the point $x$ for each $i \in I(x)$, then $\hat{f}_i$ are regular at the point $m(x-x_0)$, and thus the equality holds in \eqref{eq:pointwise}.
\end{proof}

Let $X_1$ and $X_2$ be real vector spaces, $C_1$ be a convex subset of $X_1$ such that $C_1$ is symmetric with respect to a point $x^0_1 \in C_1$ and let $C_2$ be a convex subset of $X_2$ such that $C_2$ is symmetric with respect to a point $x^0_2 \in C_2$. Consider the vector space $X := X_1 \times X_2$ and put $C: = C_1 \times C_2$. Let $\mu$ be the Minkowski functional of the set $C - (x^0_1, x^0_2)$. Consider a regular $\mu$-Lipschitz function
$$ \varphi : C_1 \times C_2 \rightarrow \mathbb{R}.$$
Denote 
$$\varphi_{x_1}(\cdot) := \varphi(x_1,\cdot), \ \ x_1 \in C_1,$$
$$\varphi_{x_2}(\cdot) := \varphi(\cdot,x_2), \ \ x_2 \in C_2.$$
 For the functions $\varphi_{x_1}(\cdot)$ and $\varphi_{x_2}(\cdot)$ put
$$ \partial_{x_i} \varphi(\cdot) := \{ \zeta \in \Sp (C_i-x^0_i) : \varphi_{x_i}^\circ(\cdot,v) \geq \left\langle\zeta,v\right\rangle \text{ for all } v \in \Sp (C_i-x^0_i) \}, \ i \in \{1,2\}.$$
\begin{proposition}
Let $\varphi : C_1 \times C_2 \rightarrow \mathbb{R}$ be a regular $\mu$-Lipschitz function. Then for all $(x_1,x_2) \in B_\mu((x^0_1,x^0_2),1)$ the next inclusion holds:
\begin{equation}
\partial^\mu_C \varphi(x_1,x_2) \subset \partial_{x_2} \varphi(x_1) \times \partial_{x_1} \varphi(x_2).\label{eq:partial}
\end{equation}
\end{proposition}
\begin{proof}
Let $\zeta = (\zeta_1,\zeta_2) \in \partial^\mu_C \varphi(x_1,x_2)$. It is enough to show that $\zeta_1 \in \partial_{x_2} \varphi(x_1)$. Since $\varphi$ is regular,
$$ \left\langle \zeta_1, v\right\rangle \leq \varphi^\circ((x_1,x_2),(v,0)) = \varphi^\prime((x_1,x_2),(v,0)) = \varphi_{x_2}^\prime(x_1,v) = \varphi_{x_2}^\circ(x_1,v);$$
hence, the inclusion in \eqref{eq:partial} holds.
\end{proof}


\section{The subdifferential of convex functions}\label{section_convex_subd}
In this section we combine the results from the previous sections to obtain the subdifferential calculus for a convex function with non-empty relative algebraic interior of its domain.

Let $S$ be a convex subset of a real vector space such that with $\icr S \neq \emptyset$, $\varphi: S \rightarrow \mathbb{R}$ be a convex function. Let us fix any real number $A \geq \varphi(x_0)$, where $x_0$ is an arbitrary point from $\icr S$, and consider the set $C_A$ , which is constructed in Subsection \ref{lipschitz}. Denote 
\begin{equation}
\mu := \text{ the Minkowski functional of the set } C_A - x_0, \ X_A := \Sp (C_A - x_0).\label{eq:mu_ca}
\end{equation}
Put
\begin{equation}
\lambda (x) := \text{a real number such that }\varphi\text{ is }\mu\text{-Lipschitz on }\lambda(x) (C_A - x_0) + x.\label{eq:lambda}
\end{equation}
By Theorem \ref{loc_lip}, such a real number $\lambda(x)$ exists for every $x \in \icr S$.
As in Section \ref{section_clarke}, each time when a vector topology in $X_A$ is considered, we supply every set $B_{\mu}(x,\lambda(x))$ with the following topology:
\begin{equation}
 U \text{ is open in }B_{\mu}(x,\lambda(x))\ \Leftrightarrow \ U - x \text{ is open in the vector topology considered in } X_A.\label{eq:topology_x}
\end{equation}
\subsection{Regularity of $\varphi$ and the Fenchel subdifferential}
Let $\pi:X_A \rightarrow Y$ be a quotient map. Consider the spaces
\begin{equation}
\begin{split}
 Y&:= X_A/_{\Ker \mu}, \text{ supplied with the norm } \|\pi(x)\| := \mu(x), \ \ x \in X_A,\label{eq:y}\\
\overline{Y}&:= \text{ the completion of the space }Y\text{ that is a Banach space. }
\end{split}
\end{equation}
We introduce the following commutative diagram
$$\xymatrix{ X_A \ar[d]_{m} \ar[r]^{\ \ \pi} &Y \ar[ld]^i\\ \overline{Y}}$$
where $i$ is a linear isometry and $m = i \circ \pi$.
Next, we define a family of functions $\{\psi_x(\cdot) : x \in \icr S \}$ by the rule
\begin{equation}
\psi_x(u) := \lim_{n\rightarrow \infty}\varphi(y_n), \ m(y_n - x)\rightarrow u, \ \{y_n\}_{n=1,2,\ldots}\subset \lambda(x)(C_A -x_0) + x.\label{eq:psi_x} 
\end{equation}
Every function $\psi_x$ is well-defined on $m(\lambda(x)(C_A -x_0))$, is Lipschitz and convex. Finally, we introduce the map $M : (\overline{Y})^* \rightarrow X_A^\prime$ defined by
$$ M(\zeta) := \zeta \circ m, \ \ \zeta \in \overline{Y}.$$
By Lemma \ref{propmaps}, the map $M$ is linear, injective and continuous with respect to $\sigma(X_A^*,X_A)$ and $\sigma((\overline{Y})^*,\overline{Y})$ topologies as soon as $X_A$ is supplied with a vector topology such that $\mu$ is continuous.

For a point $x \in \icr S$, let us consider the Fenchel subdifferential of the function $\psi_x$ on the set $B(0,\lambda(x))$:
\begin{equation}
\partial \psi_x(u) = \{\zeta \in (\overline{Y})^* : \psi_x(u) - \psi_x(v) \geq \left\langle \zeta, u-v \right\rangle \text{ for all } v \in B_{\mu}(0,\lambda(x))\}, \ u \in B(0,\lambda(x)).\label{eq:fenchel}
\end{equation}
and the Clarke subdifferential of the function $\psi_x$:
$$
\partial_{C}\psi_x(u) = \{\zeta \in (\overline{Y})^*: \psi^\prime_x(u,v) \geq \left\langle \zeta, v\right\rangle \text{ for all } v \in X_A\}, \ u \in B(0,\lambda(x)).
$$
As it is known (see, e.g.,  Clarke \cite[Proposition 2.2.7.]{clarke_book}), $\partial_C \psi_x (m(u-x)) = \partial \psi_x (m(u-x))$, \text{$u \in B_{\mu}(x,\lambda(x))$.} Applying the map $M$ to equality \eqref{eq:fenchel}, we obtain that
\begin{equation}\partial_C\varphi(u) = \{\zeta \in X_A^\prime : \varphi(u) - \varphi(v) \geq \left\langle \zeta, v-x \right\rangle \text{ for all } v \in B_{\mu}(x,\lambda(x))\}, \ u \in B_{\mu}(x_0,1). \label{fenchel}\end{equation}
The right side of \eqref{fenchel} can be called the Fenchel subdifferential of the function $\varphi$ at the point $x$ in our case. However, a linear functional $\zeta \in X_A^\prime$ may not be defined on all the elements $x \in \icr S$ and $u \in B_\mu(x,\lambda(x))$, \text{i.e.} the expression
$$ \left\langle \zeta, u-x \right\rangle = \left \langle \zeta, u\right\rangle - \left\langle \zeta,x\right\rangle $$
may not have a sense. To avoid such inconvenience, we use the Clarke subdifferential.

In the next proposition we show that $\varphi$ is regular.
\begin{proposition}\label{regular}
Let $S$ be a convex subset of a real vector space such that $\icr S \neq \emptyset$, $\varphi : S \rightarrow \mathbb{R}$ be a convex function. Let $X_A$ be the space defined in \eqref{eq:mu_ca}. Then for all $x \in \icr S$ and for all $v \in X_A$ the derivatives $\varphi^\circ(x,v)$ and $\varphi^\prime(x,v)$ exist and coincide.
\end{proposition}
\begin{proof}
Let $x \in \icr S$ and $\psi_x$ be the function defined in \eqref{eq:psi_x}. It follows from the definition of the function $\psi_x$ that for every $x \in \icr S$
$$ \psi_x^\prime(m(u-x),m(v)) = \varphi^\prime(x,v), \ u \in B_\mu(x,\lambda(x)), \ v \in X_A,$$
and since $\psi_x^\prime(m(u-x),m(v)) = \psi_x^\circ(m(u-x),m(v)) = \varphi^\circ(x,v)$ (see Lemma \ref{equality_of_gdd}), we obtain that
$$\varphi^\prime(x,v) = \varphi^\circ(x,v),  \ v \in X_A.$$
Thus $\varphi$ is regular on $\icr S$.
\end{proof}
By Proposition \ref{regular}, the Clarke subdifferential of the function $\varphi$ may be written in the form
\begin{equation}
 \partial_C^\mu\varphi(x) = \{\zeta \in X_A^\prime : \varphi^\prime(x,v) \geq \left\langle \zeta, v \right\rangle \text{ for all } v \in X_A\}, \ \ x \in \icr S.\label{eq:clarke_c}
\end{equation}
In the next proposition we establish the ``continuity'' property of $\varphi^\prime(\cdot,\cdot)$. 
\begin{proposition}
Let $S$ be a convex subset of a real vector space such that $\icr S \neq \emptyset$, $\varphi : S \rightarrow \mathbb{R}$ be a convex function. Let $X_A$ be the space defined in \eqref{eq:mu_ca}. Suppose that the space $X_A$ is supplied with a vector topology such that $\mu$ defined in \eqref{eq:mu_ca} is continuous, and let every set $B_\mu(x,\lambda(x))$ be supplied with the topology defined in \eqref{eq:topology_x}, where $\lambda(x)$ is defined in \eqref{eq:lambda}. Then for all $x \in \icr S$ the restriction of the function $\varphi^\prime(\cdot,\cdot)$ to the set $B_{\mu}(x,\lambda(x)) \times X_A$ is continuous.
\end{proposition}
\begin{proof}
Let $x \in \icr S$. By Proposition \ref{upsgdd}, the function $\varphi^\circ(\cdot,\cdot)$ is upper semi-continuous in \text{$B_{\mu}(x,\lambda(x)) \times X_A$}. By Proposition \ref{regular},  $\varphi^\prime(\cdot,\cdot)$ is upper semi-continuous in $B_{\mu}(x,\lambda(x))\times X_A$. Since $\varphi^\prime(\cdot,\cdot)$ is linear in the second argument, it is continuous. Indeed, let $\{(u_\alpha, v_\alpha)\}_{\alpha \in A} \subset B_{\mu}(x,\lambda(x))\times X_A$ be a net converging to $(u,v) \in B_{\mu}\times X_A$. Then
$$\liminf_{\substack{u_\alpha \rightarrow u,\\ v_\alpha \rightarrow v}}\varphi^\prime(u_\alpha,v_\alpha) = - \limsup_{\substack{u_\alpha \rightarrow u,\\ v_\alpha \rightarrow v}}\varphi^\prime(u_\alpha,-v_\alpha) \geq -\varphi^\prime(u,-v) = \varphi^\prime(u,v).\qedhere$$
Thus $\varphi^\prime(\cdot,\cdot)$ is continuous in $B_{\mu}(x,\lambda(x)) \times X_A$.
\end{proof}
\subsection{Some properties of $\partial^\mu_C\varphi(\cdot)$}
In the next proposition it is gathered some properties of the Clarke subdifferential.
\begin{proposition}
Let $\partial^\mu_C \varphi (\cdot) : \icr S \rightrightarrows X_A^\prime$ be the multivalued map defined in \eqref{eq:clarke_c}. Then $\partial^\mu_C \varphi(\cdot)$ has non-empty convex values. If additionally the space $X_A$ defined in \eqref{eq:mu_ca} is supplied with a vector topology such that $\mu$ defined in \eqref{eq:mu_ca} is continuous, then $\partial^\mu_C \varphi(x) \in X_A^*$ for every $x \in \icr S$ and $\partial^\mu_C \varphi(\cdot)$ has compact values with respect to $\sigma(X_A^*,X_A)$ topology.
\end{proposition}
\begin{proof}
It follows from Proposition \ref{subd_gen_properties} that $\partial^\mu_C\varphi(\cdot)$ has the mentioned properties on every set $B_\mu(x,\lambda(x))$, where $x \in \icr S$ and $\lambda(x)$ is defined in \eqref{eq:lambda}. Thus the statement holds.
\end{proof}
The next proposition describes some kind of ``upper semi-continuity'' property of $\partial^\mu_C \varphi (\cdot)$. 
\begin{proposition}\label{usc_icr}
Let $\partial^\mu_C \varphi (\cdot) : \icr S \rightrightarrows X_A^\prime$ be the multivalued map defined in \eqref{eq:clarke_c}. Suppose that $X_A$ is supplied with a vector topology such that $\mu$ is continuous and let $B_\mu(x,\lambda(x))$ be supplied with the topology defined in \eqref{eq:topology_x} for every $x \in \icr S$, where $\lambda(x)$ is defined in \eqref{eq:lambda}. Then the restriction of $\partial^\mu_C\varphi(\cdot)$ to every set $B_\mu(x,\lambda(x))$ is upper-semicontinuous with respect to the topology $\sigma(X_A^*,X_A)$ and the topology in $B_\mu (x,\lambda(x))$.
\end{proposition}
\begin{proof}
It follows from Proposition \ref{usc} that for every $x \in \icr S$ the restriction of the multivalued map $\partial^\mu_C\varphi(\cdot)$ to the set $B_\mu(x,\lambda(x))$ is upper-semicontinuous. Thus the statement holds.
\end{proof}
\begin{theorem}\label{ferma_global}
(Fermat's rule) Let $S$ be a subset of a real vector space such that $\icr S \neq \emptyset$ and let \text{$\varphi :S\rightarrow \mathbb{R}$} be a convex function. Let $C_A$ be the set defined in \eqref{eq:c_a} and $x_0$ be a point such that $C_A-x_0 = -(C_A - x_0)$.
Next, suppose that there exists $\varepsilon > 0$ and a point $u \in S$ such that \text{$\varepsilon(C_A - x_0) + u \subset S$} and the restriction of $\varphi$ to $\varepsilon(C_A - x_0) + u$ attains a minimum or a maximum at the point $u$. Then $\varphi^\prime(u,v) = 0$ for all $v \in X_A$ and $\partial^\mu_{C} \varphi(u) \ni 0$.
\end{theorem}
\begin{proof}
The proof follows from Proposition \ref{fermat}.
\end{proof}
It is shown in Counterexample \ref{cntx_fermat} that the inclusion $\partial^\mu_C \varphi(u)\ni 0$ doesn't guarantee that the point $u$ is a global minimum point of the convex function $\varphi$. 
\section{Examples and counterexamples}\label{examples}
In Subsection \ref{subs_clarke} it is provided a simple illustration of the main idea. In Subsection \ref{examples_calculus} there are examples of the calculus of $\partial^\mu_C$. Subsection \ref{c_d} concerns a few counterexamples to \text{Theorem \ref{th_main}.}
\subsection{An illustration of the main idea}\label{subs_clarke}

\begin{example}\label{example_conv}
Let $L_2[0,1]$ be the standard space of all square-integrable functions with respect to the Lebesgue measure on $[0,1]$. Let us consider the following function
\begin{equation}\varphi(x) := \begin{cases} \int_0^1 \frac{(x(t) - t)^2 }{t}dt \ \ \ \text{if }x \geq -1 \ \text{a.e.} \\ +\infty  \ \ \ \text{otherwise},
\end{cases}\label{example1}\end{equation}
where a.e. means almost everywhere and $x \in L_2[0,1]$. 
It is clear that the function $\varphi$ 
is a non-negative convex function and attains its global minimum at the function $f \in L_2[0,1]$ defined by
\begin{equation}\label{minfunction}
f(t) := t, \ \ t \in [0,1].
\end{equation}
However, $\Int \Dom \varphi = \emptyset$, since $\varphi$ attains $+\infty$, for example, at every function of the form
\begin{equation}
 \alpha \chi_{[0,\beta]} (t)  := \begin{cases} \alpha, \ \ t \in [0,\beta] \\ 0, \ \ t \notin [0,\beta],\end{cases} \label{eq:chi}
\end{equation}
where $\alpha$ is an arbitrary non-zero real number and $\beta$ is an arbitrary number from $(0,1]$. Therefore, there are no guarantees that the Fenchel subdifferential or the Clarke subdifferential of the function $\varphi$ is non-empty. Furthermore, $\icr \Dom \varphi = \emptyset$, which is perhaps less obvious and will be carefully discussed in Proposition \ref{emptyicr}. Shortly saying, the reason is in functions from $L_2[0,1] \setminus L^\infty [0,1]$, where $L^\infty [0,1]$ is the space of essentially bounded measurable functions defined on $[0,1]$; therefore, we will consider a narrowing of the domain of the function $\varphi$ to the set
\begin{equation}
S := L^\infty[0,1] \cap \Dom \varphi.\footnote{One may consider, e.g., the set of polynomials that has a root at the point $0$ and are satisfying the condition $x \geq -1$ a.e. However, we are interested in finding as bigger restriction as possible.}\label{the_set_s}
\end{equation}
Note  that $\Int S = \emptyset$ even in $L^\infty$-topology (due to the functions defined in \eqref{eq:chi}). However, as it will be shown in Lemma \ref{nonemptys}, $f \in \icr S$, where $f$ is defined in \eqref{minfunction}. This inclusion is enough to use the developed theory, i.e. to find a non-empty subdifferential and a Minkowski functional that yields Lipschitz property. This will be done in Proposition \ref{lipschlemma} and in Proposition \ref{ex1_cl}.
\end{example}
\begin{proposition}\label{emptyicr}
Let $\varphi$ be the function defined in \eqref{example1}. Then $\icr \Dom \varphi = \emptyset$.
\end{proposition} 
\begin{proof}
Assume, to the contrary, that there exists a function $x$ that belongs to $\icr \Dom \varphi$. Let $dx$ be the Lebesgue measure on $[0,1]$. Since $x$ is measurable, there exists $C > 0$ such that $$dx(\{t\in[0.5,1] : |x(t) | \leq C\})>0.$$ Denote 
$$ A:= \{t \in [0.5,1] : |x(t) | \leq C\}.$$
Since $dx(A) > 0$, there exists a density point $s \in A$ (see, e.g., Bogachev \cite[p. 366]{bogachev_measure}), which has the following property: for every neighbourhood $U_s$ of the point $s$ the inequality $dx(U_s \cap A) > 0$ holds. It is easy to check that the function
\begin{equation}
 [0,1] \ni t \mapsto t + \frac{\sqrt t}{\sqrt[4]{|s-t|}} \in \mathbb{R}\cup\{+\infty\}  \label{eq:st}
\end{equation}
belongs to  $\Dom \varphi$. Since $x \in \icr \Dom \varphi$, there exists $\alpha > 0$ such that 
\begin{equation} x(t) - \alpha \left(t +\frac{\sqrt t}{\sqrt[4]{|s-t|}}\right) \geq -1 \ \text{ a.e. }\label{eq:ineq}\end{equation}
Since the function in \eqref{eq:st} is continuous, there exists an open neighbourhood $U_s$ of the point $s$ such that
$$-1 + \alpha\left(t +\frac{\sqrt t}{\sqrt[4]{|s-t|}}\right) > C \ \text{ for all } t \in U_s. $$
Therefore, for almost all $t \in A \cap U_s$ the inequality $x(t) > C$ holds. However, $dx(A\cap U_s) > 0$, hence the inequality in \eqref{eq:ineq} fails on the set of positive measure. Contradiction.
\end{proof}

\begin{lemma}\label{nonemptys}
Let $S$ be the set defined in \eqref{the_set_s} and let $f$ be the function defined in \eqref{minfunction}. Then $f \in \icr S$.
\end{lemma}
\begin{proof}
Since $S$ is convex and contains $0$, every element from $\Sp S$ can be represented as 
$$\alpha v - \beta u,\ \text{ where }\alpha >0,\ \beta > 0 \text{ and }  \{v,u\} \subset S.$$
Let $\|\cdot \|_\infty$ denotes the norm in $L^\infty[0,1]$ and let $\varepsilon$ be an arbitrary number from the interval \text{$(0, (\alpha + \beta \|u\|_\infty)^{-1})$.} Then
$$f(t) + \varepsilon (\alpha v(t) -\beta u(t)) \geq \varepsilon(-\alpha - \beta \|u\|_\infty) > -1 \ \text{ for almost all } t \in [0,1].$$
Thus $f \in \icr S$.
\end{proof}
In the next proposition we are to illustrate the results of Section \ref{section_convex_case}. We construct a Minkowski functional that yields a locally Lipschitz property for the function $\varphi$. According to the procedures considered in Section \ref{section_convex_case}, a Minkowski functional can be constructed straightforward from the convex function.

\begin{proposition}\label{lipschlemma}
Let $S$ be the set defined in \eqref{the_set_s} and let $\varphi$ be the function defined in \eqref{example1}. Then $\varphi$ is locally $\mu$-Lipschitz and regular on $\icr S$, where $\mu$ is a norm in $\Sp S$ defined by
\begin{equation}\label{normaexample}
\mu(x):=\sqrt{\int_0^1 \frac{(x(t))^2}{t}dt,} \ \ x \in \Sp S. 
\end{equation}
\end{proposition}
In case of Proposition \ref{lipschlemma} the locally $\mu$-Lipschitz property means that for every point $x_0 \in \icr S$ there exists $L > 0$ and $\lambda > 0$ such that for all pairs $\{x,y\}\subset B_{\mu}(x_0,\lambda)$ the following inequality holds:
\begin{equation}\label{super_inequality}
\left|\int^1_0 \frac{(x(t)-t)^2}{t}dt - \int^1_0\frac{(y(t)-t)^2}{t}dt \right|\leq L\sqrt{\int_0^1\frac{(x(t)-y(t))^2}{t}dt}\end{equation}
\begin{proof}
Let $f:[0,1]\rightarrow[0,1]$ be the function defined by $f(t) := t$.
Consider the following set:
$$S_1 := \{ x \in S : \varphi(x) \leq 1\} = \{x \in S : \int_0^1 \frac{(x(t)-f(t))^2}{t}dt\leq 1\}.$$
The set $S_1 - f$ is symmetric by the definition of $S_1$. Let $\mu$ be the Minkowski functional of the set $S_1 -f $. Then
$$\mu(x) = \inf \{\alpha > 0 : x \in \alpha (S_1 - f)\} = \inf\{ \alpha > 0 : \int_0^1\frac{(x(t))^2}{t} \leq \alpha^2\}=\sqrt{\int_0^1\frac{(x(t))^2}{t}dt}.$$
By Theorem \ref{lipschitz}, the function $\varphi$ is locally $\mu$-Lipschitz on the set $\icr S$. Since the function $\varphi$ is convex, it is regular by Proposition \ref{regular}. Easy calculations show that $B_\mu(0,1)$ is absorbing in $\Sp S$ and hence $\mu$ is a norm in $\Sp S$.
\end{proof}
Let us consider the Clarke subdifferential $\partial_C^\mu \varphi(\cdot)$ of the function $\varphi$ with respect to the Minkowski functional $\mu$:
\begin{equation}\partial_C^\mu \varphi(x) = \{ \zeta \in (\Sp S)^\prime : \varphi^\prime(x,v) \geq \left\langle \zeta,v \right\rangle\text{ for all } v \in \Sp S\}, \ \ x \in \icr S.\label{eq:clarke_ex}\end{equation}
\begin{proposition}\label{ex1_cl}
Let $\varphi$ be the function defined in \eqref{example1} and let $\partial_C^\mu\varphi(\cdot)$ be its Clarke subdifferential defined in \eqref{eq:clarke_ex}. Then
\begin{equation}
\partial^\mu_C\varphi(x) = \left\{\Sp S \ni v \mapsto 2 \int_0^1\frac{(x(t)-t)}{t}v(t) dt \in \mathbb{R}\right\}.\label{eq:ex1_clarke}
\end{equation}
Particularly, $0 \in \partial_C^\mu\varphi(f)$, where $f$ is the function defined in \eqref{minfunction}.
\end{proposition}
\begin{proof}
Let us calculate the directional derivative $\varphi^\prime(\cdot,\cdot)$. Let $x \in \icr S$ and $v \in \Sp S$. The existence of $\varphi^\prime(x,v)$ is guaranteed, e.g., by Proposition \ref{regular}. Then
$$
\varphi^\prime(x,v) = \lim_{\alpha \rightarrow 0+} \frac{1}{\alpha} \int_0^1 \frac{(x(t)+\alpha v(t) - t)^2 - (x(t) -t)^2}{t}dt =$$$$= \lim_{\alpha \rightarrow 0+} \frac{1}{\alpha} \int_0^1 \frac{\alpha^2 (v(t))^2 +2\alpha v(t)(x(t) -t)}{t}dt = 2\int_0^1\frac{(x(t)-t)}{t}v(t)dt.$$
Since $\varphi(\cdot, \cdot)$ is linear in the second argument, $\partial^\mu_C\varphi(\cdot)$ consists of only one element at every point from $\icr S$. Therefore, the equality holds in \eqref{eq:ex1_clarke}. In particular $0 \in \partial_C^\mu\varphi(f)$.\qedhere
\end{proof}
\subsection{A few illustrations of the calculus}\label{examples_calculus}
The following proposition is an illustration of Theorem \ref{lebourg}.
\begin{proposition}\label{ex_lebourg}
Let $\varphi$ be the function defined in \eqref{example1}, $S$ be the set defined in \eqref{the_set_s} and $\mu$ be the Minkowski functional defined in \eqref{normaexample}. Let $x$ and $y$ be two distinct points from $\icr S$ such that there exists a ball $B_\mu(x_0,\varepsilon)$ that consists both the points $x$ and $y$ and such that $\varphi$ is $\mu$-Lipschitz on $B_\mu(x_0,\varepsilon)$. Then there exists $\alpha \in (0,1)$ such that
\begin{equation}
\varphi(x) - \varphi(y) = 2\int_0^1\frac{(\alpha x(t)+(1-\alpha)y(t)-t)}{t}(x(t) - y(t))dt.\label{eq:lebourg}
\end{equation}
\end{proposition}
\begin{proof}
Since the function $\varphi$ is $\mu$-Lipschitz on the set $B_\mu(x_0,\varepsilon)$ and the line segment with endpoints $x$ and $y$ is a subset of $B_\mu(x_0,\varepsilon)$, equality \eqref{eq:lebourg} follows from \eqref{eq:ex1_clarke} and Theorem \ref{lebourg}.
\end{proof}
The following example is a modification of Example \ref{example_conv} and serves as an illustration for \text{Theorem \ref{prop_chain_rule2}} (Chain rule 2).
\begin{example}\label{ex_chain_rule2}
Let $\varphi$ be the function defined in Example \ref{example_conv}, $S$ be the set defined in \eqref{the_set_s}, $\mu$ be the Minkowski functional defined in \eqref{normaexample}. Consider the function $g:S\rightarrow\mathbb{R}$ defined by
\begin{equation}
g(x) := e^{\varphi(x)}, \ \ x \in S.\label{eq:gg}
\end{equation}
The function $g$ is locally $\mu$-Lipschitz on the set $\icr S$ and moreover its Clarke subdifferential can be calculated as
\begin{equation}
\partial^\mu_C g(x) =  e^{\varphi(x)} \partial^\mu_C\varphi(x),\label{eq:ex2}
\end{equation}
i.e., by Proposition \ref{ex1_cl},
\begin{equation}
\partial^\mu_C g(x) = \left\{\Sp S \ni v \mapsto 2 e^{\varphi(x)}\int_0^1\frac{(x(t)-t)}{t}v(t) dt \in \mathbb{R}\right\}.\label{eq:clarkeg}
\end{equation}
Indeed, by Theorem \ref{prop_chain_rule2},
$$\partial^\mu_C g(x) \subset \{ \alpha \zeta : \zeta \in \partial^\mu_C\varphi(x), \ \alpha \in \partial_C e^{(\cdot)} (\varphi(x))\} = \{e^{\varphi(x)} \zeta : \zeta \in \partial^\mu_C\varphi(x)\} = e^{\varphi(x)} \partial^\mu_C\varphi(x).$$
In this example there is no need to specify a vector topology in $\Sp S$. By Corollary \ref{cor_chain_rule1}, the equality holds in \eqref{eq:ex2}.
\end{example}

The next examples are a combination of Example \ref{example_conv} and Example \ref{ex_chain_rule2} and illustrate a use of Theorem \ref{sumrule} and Theorem \ref{multiplerule}. Firstly we need to show that the function $g$ defined in \eqref{eq:gg} is convex and regular.
\begin{lemma}\label{lemma_greg}
Let $g$ be the function defined in \eqref{eq:gg}. Then the function $g$ is convex and regular.
\end{lemma}
\begin{proof}
Let $\lambda \in [0,1]$ and $\{x,y\} \subset S$. Then
$$ g(\lambda x + (1-\lambda) y ) = e^{\varphi(\lambda x + (1-\lambda) y)} \leq e^{\lambda\varphi(x) + (1-\lambda)\varphi(y)} \leq \lambda e^{\varphi(x)} + (1-\lambda) e^{\varphi(y)} = \lambda g(x) + (1-\lambda) g(y).$$
Thus the function $g$ is convex. Since $g$ is a $\mu$-Lipschitz function, it is regular by Proposition \ref{regular}.
\end{proof}
\begin{example}\label{ex_sum_rule}
Let $\varphi$ be the function defined in Example \ref{example_conv}, $g$ be the function considered in \text{Example \ref{ex_chain_rule2}}, $S$ be the set defined in \eqref{the_set_s} and $\mu$ be the Minkowski functional defined in \eqref{normaexample}. Then
$$\partial^\mu_C (g+\varphi) (x) = \left\{\Sp S \ni v \mapsto 2 \left(1+e^{\varphi(x)}\right)\int_0^1\frac{(x(t)-t)}{t}v(t) dt \in \mathbb{R}\right\}, \ \ x \in \icr S.$$
Indeed, let $x \in \icr S$. By Lemma \ref{lemma_greg} and by Proposition \ref{lipschlemma}, the functions $g$ and $\varphi$ are regular on the set $\icr S$. Using Corollary \ref{corsumrule}, we see that
\begin{equation}
 \partial^\mu_C (g+\varphi) (x) = \partial^\mu_C g(x) + \partial^\mu_C \varphi(x).\label{eq:sumgfi}
\end{equation}
The Clarke subdifferential $\partial^\mu_C g(x)$ is calculated in \eqref{eq:clarkeg} and the Clarke subdifferential of the function $\varphi$ is calculated in \eqref{eq:ex1_clarke}. Since both the sets in the right side of \eqref{eq:sumgfi} are convex, 
$$\partial^\mu_C (g+\varphi) (x) = \left(1+ e^{\varphi(x)}\right) \partial^\mu_C\varphi(x),$$
i.e. 
$$\partial^\mu_C (g+\varphi) (x) = \left\{\Sp S \ni v \mapsto 2 \left(1+e^{\varphi(x)}\right)\int_0^1\frac{(x(t)-t)}{t}v(t) dt \in \mathbb{R}\right\}.$$
\end{example}
\begin{example}\label{ex_mul_rule}
Let $\varphi$ be the function defined in Example \ref{example_conv}, $g$ be the function considered in Example \ref{ex_chain_rule2}, $S$ be the set defined in \eqref{the_set_s} and $\mu$ be the Minkowski functional defined in \eqref{normaexample}. Then
$$\partial^\mu_C (g\varphi) (x) = \left\{\Sp S \ni v \mapsto 2 (1+\varphi(x))e^{\varphi(x)}\int_0^1\frac{(x(t)-t)}{t}v(t) dt \in \mathbb{R}\right\}, \ \ x \in \icr S.$$
Indeed, it follows from Lemma \ref{lemma_greg} and Proposition \ref{lipschlemma} that both the functions $g$ and $\varphi$ are regular on the set $\icr S$; moreover, both $g$ and $\varphi$ are non-negative at every point $x \in \icr S$. \text{By Corollary \ref{cor_multiplerule}}, 
$$\partial^\mu_C(g \varphi)(x) = \varphi(x)\partial^\mu_Cg(x) + g(x) \partial^\mu_C \varphi(x).$$
Using \eqref{eq:ex2} and \eqref{ex1_cl},
\begin{equation}
\begin{split}
\partial^\mu_C(g\varphi)(x) &= e^{\varphi(x)}(1+\varphi (x))\partial^\mu_C\varphi(x) =\\
&= \left\{\Sp S\ni v \mapsto 2(1+\varphi(x))e^{\varphi(x)}\int_0^1\frac{x(t) - t}{t} v(t) dt\right\}. 
\end{split}
\end{equation}
\end{example}
The following example illustrates a use of Theorem \ref{prop_chain_rule1}.
\begin{example}\label{ex_chain_rule1}
Let $S$ be the set defined in \eqref{the_set_s} and let the space $\Sp S$ be supplied with a normed structure induced from the space $L_2[0,1]$. Let $\varphi$ be the function defined in \eqref{example1} and $f$ be the function defined in \eqref{minfunction}. Consider the function $g:\Sp S \rightarrow \Sp S$ defined by $g(x) = fx$, $x \in \Sp S$, i.e. $g(x)(t) = f(t)x(t)$. It is clear that $g(S) \subset S$.
The function $g$ satisfies the next inequalities:
\begin{equation}\label{eq:lg}
\begin{split}
\mu(g(x) - g(y))&\leq \|x-y\|,\\
\|g(x)-g(y)\|&\leq \|x-y\|,
\end{split}
\end{equation}
where $\mu$ is defined in \eqref{normaexample} and $\|\cdot\|$ is the norm in $L_2[0,1]$. Indeed,
$$\mu(g(x) - g(y)) = \sqrt{\int_0^1 \frac{(tx(t)-ty(t))^2}{t}dt} = \sqrt{\int_0^1t (x(t)-y(t))^2 dt} \leq \|x-y\|;$$
$$\|g(x) - g(y) \| = \sqrt{\int_0^1 t^2(x(t)-y(t))^2 dt} \leq \|x-y\|.$$
The G\^ateux derivative of the function $g$ is given by $Dg(x)v = fv = g(v)$, hence $Dg(x)$ doesn't depend on $x$ and therefore $x \mapsto Dg(x)$ is continuous. Furthermore, $v \mapsto g(v)$ is continuous \text{due to \eqref{eq:lg}.} Thus Theorem \ref{prop_chain_rule1} can be applied to the composition $\varphi \circ g$:
\begin{equation}
\partial^\mu_C(\varphi\circ g)(x) \subset \{\zeta \circ g(\cdot) : \zeta \in \partial^\mu_C\varphi(x)\} = \left\{\Sp S \ni v \mapsto 2 \int_0^1(x(t)-t)v(t) dt \in \mathbb{R}\right\}, \ \ x \in \icr S,
\end{equation}
where $\partial^\mu_C\varphi(x)$ is calculated in Proposition \ref{ex1_cl}. The function $\varphi \circ g$ is locally $\|\cdot\|$-Lipschitz, which follows from \eqref{eq:lg} and Proposition \ref{lipschlemma}:
$$|\varphi(g(x)) - \varphi(g(y)) | \leq L \mu(g(x) - g(y)) \leq L \|x - y\|.$$

Therefore, $\partial^\mu_C(\varphi\circ g)(x)$ is non-empty for every $x \in \icr S$ and thus
$$
\partial^\mu_C(\varphi\circ g)(x) =  \left\{\Sp S \ni v \mapsto 2 \int_0^1(x(t)-t)v(t) dt \in \mathbb{R}\right\}, \ \ x \in \icr S.
$$
\end{example}
The modified theory does not guarantee that if a subdifferential of a convex function consists zero at some point, then this point is a global minimum point. This circumstance is illustrated in the following counterexample.
\begin{cntex}\label{cntx_fermat}
Let us consider the function $\varphi : \mathbb{R}^2 \rightarrow \mathbb{R}$ defined by $\varphi (x,y) = x^2 + y^2$
and the Minkowski functional $\mu(x,y) := |x|.$
Then $\varphi$ is locally $\mu$-Lipschitz and its Clarke subdifferential with respect to $\mu$ consists zero at every point $(0,y) \in \mathbb{R}^2$. Indeed, the function $x \mapsto \varphi(x,y)$ is convex in $\mathbb{R}$ and since $\mathbb{R}$ has finite dimension, the function $x \mapsto \varphi(x,y)$ is locally Lipschitz in $\mathbb{R}$. By our definition of $\mu$-locally Lipschitz property, the function $\varphi(\cdot,\cdot)$ is locally $\mu$-Lipschitz in $\mathbb{R}^2$. However, 
$$ \varphi^\prime((0,y),(v,0)) = 0 \text{ for every } v \in \mathbb{R},$$
thus $\partial_C^\mu\varphi(0,y) \ni 0$ for every $y \in \mathbb{R}$.
\end{cntex}
\subsection{Counterexamples to extending the $\mu$-Lipschitz property}\label{c_d}
The following counterexample shows that an ``$\varepsilon$-step'' from the ``boundary'' of the set $C$ in \text{Theorem \ref{th_main}} is necessary.
\begin{cntex} \label{cntex_epsilon}
The conclusion of Theorem \ref{th_main} cannot be weakened to the following relation:
\begin{equation}
\text{The function }\varphi\text{ is }\mu\text{-Lipschitz on the set }B_\mu(p,1).\label{str_lip}
\end{equation}
Indeed, put $X:=\mathbb{R}$, $C:=[-1,1]$, $\varphi(x):=-\sqrt{1 - |x|}$. Then $p = 0$ and $\mu (\cdot) = |\cdot|$.
Suppose that relation \eqref{str_lip} holds, and let $L>0$ be a $\mu$-Lipshitz constant of the function $\varphi$ on the set $B_\mu(0,1)$. It follows from the Mean value theorem on the segment $[1-1/n, 1]$ that there exists $\xi \in (1-1/n, 1)$ such that
$$|\varphi(1) - \varphi(1-1/n)|n  = \varphi^\prime (\xi),$$
hence $L \geq \varphi^\prime(\xi)$. However, it follows from the estimation
$$ \varphi^\prime (\xi) = \frac{1}{2\sqrt{1-\xi}} \geq \frac{\sqrt{n}}{2}.$$
that $L \geq \sqrt{n}/2$ for each natural number $n$. Contradiction.
\end{cntex}
Recall that a function $\varphi$ is called \textit{quasiconvex} on a convex set $C$ if for every pair $\{u,v\} \subset C$ and every number $\alpha \in [0,1]$ the next inequality holds:
$$\varphi(\alpha u + (1-\alpha) v) \leq \max\{\varphi(u),\varphi(v)\}. $$
\begin{cntex}\label{cntex_convexity}
The condition of convexity of the function $\varphi$ in Theorem $\ref{th_main}$ cannot be weakened to quasiconvexity.
Indeed, suppose that the conclusion of Theorem \ref{th_main} holds for quasiconvex functions.
Put $X:=\mathbb{R}$, $C := [-1,1]$, $\varphi(x) := [x]$ (an integer part of $x$). Then $p = 0$ and $\mu (\cdot) = |\cdot|$. Let $\varepsilon \in (0,1)$ be fixed and let $L>0$ be a $\mu$-Lipschitz constant from the theorem. Choose a natural number $N$ such that for all $n \geq N$ the inequality $1/n \leq \varepsilon$ holds. Then
$$ 1 = |\varphi(-1/n) - \varphi(1/n)| \leq \frac{2L}{n} \rightarrow 0, \, n \rightarrow \infty. $$
Contradiction.
\end{cntex}
\begin{cntex}\label{cntex_yravn}
The condition of symmetry of the set $C$ in Theorem \ref{th_main} cannot be omitted.
Indeed, suppose that Theorem \ref{th_main} holds also for a convex set $C$ such that there is no point $p \in C$ that yields the equality $C-p = -(C-p)$. There are at least two possible choices how to understand the conclusion in this case:
\begin{enumerate}[(a)]
\item There exists a point $p \in C$ such that for every $\varepsilon \in (0,1)$ the function $\varphi$ is $\mu$-Lipschitz on the set $\varepsilon(C-p) + p$, where $\mu$ is the Minkowski functional of the set $C-p$;\label{c_a}
\item For every $\varepsilon \in (0,1)$  the function $\varphi$ is $\nu$-Lipschitz on the set $\varepsilon(C-C)$, where $\nu$ is the Minkowski functional of the set $C-C$.\label{c_b}
\end{enumerate}
Put $X := \mathbb{R}$, $C:=[-1,+\infty)$, $\varphi(x):=-x$ and let $\varepsilon \in (0,1)$ be given.
In case of choice \eqref{c_a}, for any point $p \in C$ the Minkowski functional $\mu$ of the set $C-p$ may be written in the following form:
$$ \mu(x) = \begin{cases}
0,\  x \geq 0 \\
-x,\  x < 0
\end{cases}$$
and $B_{\mu}(p,\varepsilon) \supset [0,+\infty]$. If $L$ is a $\mu$-Lipschitz constant, then 
$$1 = |\varphi(1) - \varphi(0)|  \leq L \mu_{C}(1) = 0,$$
which yields a contradiction.
In case of choice \eqref{c_b}, the Minkowski functional $\nu$ of the set $C-C$ is equal to $0$ since $C-C = \mathbb{R}$. Thus this choice also fails.
\end{cntex}
Note that the function $\varphi$ that is considered in Counterexample \ref{cntex_yravn} is Lipschitz with respect to the absolute value $|\cdot|$.

\end{document}